\documentclass[11pt, reqno]{amsart}
\usepackage{hyperref}
\hypersetup{colorlinks=true, linkcolor=blue, citecolor=red, urlcolor=blue}
\usepackage{amssymb,amsmath,amsfonts,latexsym,amscd,amsthm}
\usepackage[mathscr]{eucal}
\usepackage{csquotes}
\usepackage{xfrac}
\usepackage{bigints}
\usepackage{bm}
\usepackage{pdfpages}
\usepackage[utf8]{inputenc}
\usepackage{mathtools}
\allowdisplaybreaks
\usepackage{enumerate}
\usepackage{fancyhdr}
\usepackage{array,graphics,color}
\numberwithin{equation}{section}
\theoremstyle{plain}
\newtheorem{dfn}{Definition}[section]
\newtheorem{lma}[dfn]{Lemma}
\newtheorem{nota}[dfn]{Notations}
\newtheorem{prb}[dfn]{Problem}

\newtheorem{crl}[dfn]{Corollary}

\newtheorem{rmrk}[dfn]{Remark}
\newtheorem{thm}[dfn]{Theorem}
\newtheorem{case}{Case}
\newtheorem{subcase}{Subcase}
\numberwithin{subcase}{case}
\DeclarePairedDelimiterX{\norm}[1]{\lVert}{\rVert}{#1}
\DeclarePairedDelimiterX{\bnorm}[1]{\big\lVert}{\big\rVert}{#1}
\DeclarePairedDelimiterX{\Bnorm}[1]{\Big\lVert}{\Big\rVert}{#1}

\newcommand{\R}{\mathbb{R}}
\newcommand{\N}{\mathbb{N}}
\newcommand{\C}{\mathbb{C}}
\newcommand{\Cn}{\mathbb{C}^n}

\newcommand{\D}{\mathbb{D}_{1+\epsilon}^{n}}
\newcommand{\Ra}{(-1-\epsilon, 1+\epsilon)^{n}}

\newcommand{\hil}{\mathcal{H}}

\newcommand{\hils}{\mathcal{B}_2(\hil)}
\newcommand{\boh}{\mathcal{B}_1(\hil)}

\newcommand{\bh}{\mathcal{B}(\hil)}

\newcommand{\bph}{\mathcal{B}_p(\hil)}

\newcommand{\Tr}{\operatorname{Tr}}
\newcommand{\ii}{\operatorname{i}}

\newcommand{\tr}{\operatorname{\tau}}

\newcommand{\mL}{\mathcal{L}_{1}}
\newcommand{\mLL}{\mathcal{L}_{2}}

\newcommand{\Lpt}{\mathcal{L}_{p}(\mathcal{M}, \tr)}

\newcommand{\dmds}{\dfrac{d^m}{ds^m}}
\newcommand{\dmmds}{\dfrac{d^{m-1}}{ds^{m-1}}}
\newcommand{\drds}{\dfrac{d^r}{ds^r}}
\newcommand{\drrds}{\dfrac{d^{r-1}}{ds^{r-1}}}
\newcommand{\ddt}{\dfrac{d}{dt}}

\newcommand{\cir}{\mathbb{T}}

\newcommand{\Mcal}{\mathcal{M}}

\newcommand{\A}{\textbf{A}_{n}}
\newcommand{\B}{\textbf{B}_{n}}
\newcommand{\X}{\textbf{X}_{n}}
\newcommand{\V}{\textbf{V}_{n}}

\newcommand{\mK}{\mathcal{K}}
\newcommand{\Cmn}{\normalfont{\text{Com}_{n}}}
\newcommand{\cn}{\mathcal{C}_{n}}

\begin{document}
\title[Estimates and Higher-Order Spectral Shift Measures]
{Estimates and Higher-Order Spectral Shift Measures in Several Variables}
	
\author[Chattopadhyay] {Arup Chattopadhyay}
\address{Department of Mathematics, Indian Institute of Technology Guwahati, Guwahati, 781039, India}
\email{arupchatt@iitg.ac.in, 2003arupchattopadhyay@gmail.com}

\author[Giri] {Saikat Giri}
\address{Department of Mathematics, Indian Institute of Technology Guwahati, Guwahati, 781039, India}
\email{saikat.giri@iitg.ac.in, saikatgiri90@gmail.com}

\author[Pradhan]{Chandan Pradhan}
\address{Department of Mathematics, Indian Institute of Technology Guwahati, Guwahati, 781039, India}
\email{chandan.pradhan2108@gmail.com}

	
	
\begin{abstract}
In recent years, higher-order trace formulas of operator functions have attracted considerable attention to a large part of the perturbation theory community. In this direction, we prove estimates for traces of higher-order derivatives of multivariable operator functions with associated scalar functions arising from multivariable analytic function space and, as a consequence, derive higher-order spectral shift measures for pairs of tuples of commuting contractions under Hilbert-Schmidt perturbations. These results substantially extend the main results of \cite{Sk15}, where the estimates were proved for traces of first and second-order derivatives of multivariable operator functions. In the context of the existence of higher-order spectral shift measures, our results extend the relative results of \cite{DySk09, PoSkSu14} from a single-variable to a multivariable setting under Hilbert-Schmidt perturbations. Our results rely crucially on heavy uses of explicit expressions of higher-order derivatives of operator functions and estimates of the divided deference of multivariable analytic functions, which are developed in this paper, along with the spectral theorem of tuples of commuting normal operators.
\end{abstract}
	
\maketitle
\section{Introduction}\label{Sec1}
Let $\bh$ be the the algebra of bounded linear operators acting on a complex separable Hilbert space $\hil$, endowed with the operator norm. Let $\Mcal\subseteq\bh$ be a semifinite von Neumann algebra and $\tr$ be a semifinite normal faithful trace on $\Mcal$. Note that, example of such $(\Mcal, \tau)$ is $(\bh, \Tr)$, where $\Tr$ is the canonical trace. For a pair of operators $(T,V)$ affiliated with $\Mcal$, a nice function $f:\C\to\C$, and a natural number $n$, the $n$-th order Taylor remainder is denoted by $\mathscr{R}_{n,T,f}(V)$ and is defined by
\begin{align}\label{TR}
\mathscr{R}_{n,T,f}(V):=f(T+V)-\sum_{k=0}^{n-1}\frac{1}{k!}\,\dfrac{d^k}{ds^k}\bigg|_{s=0}f(T+sV).
\end{align} 

In perturbation theory, one is always interested in the behavior of the Taylor remainders for the  functional $V\mapsto\tr(f(T+V)),$ where $T$ is either self-adjoint operator affiliated with $\Mcal$ or a bounded operator in $\Mcal$, $V\in\Mcal$ and $f:\C\to\C$ a nice function so that the above functional make sense. It is interesting to note that, the behavior of the first-order Taylor remainder of the above functional is described by a complex Borel measure $\mu$ on a suitable subset $\Omega\subseteq\C$ satisfying
\begin{align}\label{TR1}
\tr(\mathscr{R}_{1,T,f}(V))=\int_{\Omega}f'(z)d\mu(z).
\end{align}
The measure $\mu$ is known as the first-order spectral shift measure (SSM). The spectral shift measure in \eqref{TR1} corresponding to a pair of self-adjoint operators originated from the foundational work of M. G. Krein \cite{Krein53} following Lifshit's work \cite{Lf52}, where the authors shown that the SSM is absolutely continuous with respect to the Lebesgue measure. After that, unceasing progress has been made over the years on the existence of first-order spectral shift measures for single-variable operator functions associated with both self-adjoint operators and contractions. We refer to \cite{MN15,MNP19,Rybkin87,Rybkin89,Rybkin94} for more on the existence of first-order spectral shift measures for pairs of arbitrary contractions. The trace of the second-order Taylor remainders were first proposed by Koplienko \cite{Kp84} and also he conjectured the existence of higher-order absolutely continuous spectral shift measures with respect to the Lebesgue measure, which is resolved affirmatively in \cite{PoSkSu13}.  For more on the higher-order trace formulas corresponding to pairs of self-adjoint operators and contractions we refer to \cite{Peller06,DySk09, PoSkSu13,PoSkSu14, Sk10}, the survey article \cite{Sksurvey} and the references cited therein.

In contrast with single-variable operator functions, multivariate operator functions are much more complicated to understand. Articles  \cite{Agler12,Ha05,HaPed03} contribute to a better understanding of multivariate operator functions. Although the trace of Taylor-like approximations is well-studied for single-variable functions $f$, but they have not yet been explored for multivariate functions. In this direction, articles \cite{Sk15} and \cite{CGP23} extended the results on Taylor-like approximations from single-variable operator functions to the multivariate operator functions, where only first and second-order Taylor remainders were considered. For more on these discussions, we need to fix the following notations.

\begin{nota}\label{Notations}
Assume all operators in $\mathcal{M}$.
\begin{itemize}		
\item[(i)] Let $m, n\in\N$ with $n\geq 2$.
\item[(ii)] Denote $\Cmn$ as the set of $n$-tuples of pairwise commuting elements in $\mathcal{M}$ and by $\cn(\subseteq\Cmn),$ a set consisting of $n$-tuples of contractions.
\item[(iii)] Let $\A=(A_{1},\ldots,A_{n})$, $\B=(B_1,\ldots,B_n),~\V=(V_1,\ldots,V_n)$, and\\ $\bm{X}_{n}(s)$ $=(X_{1}(s),\ldots,X_{n}(s))$ for $s\in[0,1]$, where $V_{j}=B_{j}-A_{j},\text{ and }X_{j}(s)=A_{j}+sV_{j},~1\le j\le n$.  
\item[(iv)] Assume that $k_{1},\ldots,k_{n}\in\N\cup \{0\}$ and $\bm{X}_{n}=(X_{1},\ldots,X_{n})$ be a $n$-tuple of operators. Denote
\begin{align}
\nonumber&T_{k_{1},\ldots,k_{n}}(\bm{X}_{n})=X_{1}^{k_{1}}\cdots X_{n}^{k_{n}},\\
\nonumber&T_{k_{m},\ldots,k_{n}}({}_{m}\bm{X}_{n})=X_{m}^{k_{m}}\cdots X_{n}^{k_{n}},\,\,\,\,\,\,1< m\le n.
\end{align} 
\item[(v)] Let $\mathcal{N}_{n}$ denote the set of pairs ${\bm{A}_{n},\bm{B}_{n}}\in\mathcal{C}_{n}$ such that $(A_{1}+tV_{1},\ldots,A_{n}+tV_{n}),~t\in[0,1]$, consists of tuples of commuting contractions, which admits a commuting contractive normal dilation; that is, $\{\bm{A}_{n},\bm{B}_{n}\}\in\mathcal{N}_{n}$ if and only if $(A_{1}+tV_{1},\ldots,A_{n}+tV_{n})\in\mathcal{C}_{n}$ satisfies \eqref{Dilation} for every $t\in[0,1]$. More on dilation is recalled in the next section.
\end{itemize}	
\end{nota}
A necessary and sufficient condition for $(A_{1}+tV_{1},\ldots,A_{n}+tV_{n})\in\Cmn$, for every $t\in[0,1]$ is that $\bm{A}_{n}, \bm{B}_{n}, \bm{V}_{n}\in\Cmn$. An equivalent criterion for $\X(t)\in\cn$, for every $t\in[0,1]$ is established in \cite[Lemma 3.7]{Sk15}. Note that if $\bm{A}_{n}, \bm{B}_{n}\in\cn$, then $X_{j}(t)$, $1\le j\le n$, is a contraction for every $t\in[0,1]$.

Next we recall definitions and some important properties of the classical noncommutative $L_{p}$-spaces. The noncommutative $L_{p}$-space, $p\in[1,\infty)$, associated with $(\mathcal{M}, \tau)$ is defined by
\begin{align}
\nonumber&L_{p}(\mathcal{M}, \tau):=\left\{X\eta\mathcal{M}:~~\|X\|_{p}:=(\tau(|X|^{p}))^{1/p}<\infty\right\},
\end{align}
where we used the notation $X\eta\mathcal{M}$ for a closed densely defined operator $X$ affiliated with $\mathcal{M}$. Let $\Lpt:=L_{p}(\Mcal,\tr)\cap\Mcal.$ Noncommutative $\mathcal{L}_{p}$-spaces include the Schatten ideal $\bph=\mathcal{L}_{p}(\bh, \Tr).$ Often we use the phrase \enquote{$\tau$ is the standard trace} to mean $\mathcal{M}=\bh$ and $\tau$ is the canonical trace $\Tr$ acting on $\boh$. Throughout the paper, the tuples of operators $\A, \B$ consist of elements of $\mathcal{M}$ and the tuple $\V$ consists of elements from $\mathcal{L}_{p}(\mathcal{M},\tau).$
\vspace*{0.2cm}

\noindent\textbf{Function spaces.}  Let $C^{m}(\C^{n})$ denote the space of $m$-times continuously differentiable functions on $\C^{n}$. For $m=0$, we use the convention that $C^{0}(\C^{n})=C(\C^{n}),$ the space of continuous functions on $\C^{n}$. Given a function $f$ is holomorphic on
$$\D=\{(z_{1},\ldots,z_{n})\in\Cn : |z_{j}|<1+\epsilon, 1\le j\le n\},\,\,\,\,\,\epsilon>0,$$
has an absolutely convergent series in $\D$
\begin{equation}\label{srf}
f(z_{1},\ldots,z_{n})=\sum_{k_1,\ldots,k_n\ge 0}c_{k_1,\ldots,k_n}z_{1}^{k_{1}}\cdots z_{n}^{k_{n}},
\end{equation}
where $c_{k_1,\ldots,k_n}\in \mathbb{C}$. Then for a tuple of contractions $(A_{1},\ldots,A_{n})$, the \textit{multivariate operator function} is defined by
\begin{equation}\label{opf}
f(A_{1},\ldots,A_{n})=\sum_{k_1,\ldots,k_n\ge 0}c_{k_1,\ldots,k_n}A_{1}^{k_{1}}\cdots A_{n}^{k_{n}}.
\end{equation}
Denote by $\mathcal{A}(\D)$ the set of holomorphic functions on $\D$ and by $\mathcal{A}((-1-\epsilon, 1+\epsilon)^{n})$ the set of functions that can be represented by their Taylor series \eqref{srf} for $(z_{1},\ldots,z_{n})\in\Ra$. Often, we use the notation $\mathcal{A}(\C^{n})$ to mean the set of all holomorphic functions on $\Cn$. 

Assuming the above notations, the following results were obtained in \cite{Sk15}.
\begin{thm}{\normalfont\cite[Theorem 3.8, Theorem 4.8]{Sk15}}\label{Anna1}
Assume the above notations. 
\begin{itemize}
\item[(i)] Let $V_{j}\in\mathcal{L}_{1}(\Mcal,\tr),~1\le j\le n$. If $\{\A,\B\}\in\mathcal{N}_n$, then there exist finite measures $\mu_1,\ldots,\mu_n$ on $\Omega^n$ such that
\begin{align}
\label{Anna1eq1}&\tr\big(f(\B)-f(\A)\big)
=\sum_{j=1}^n\int_{\Omega^n}\frac{\partial f }{\partial z_j}(z_1,\ldots,z_n)\,d\mu_j(z_1,\ldots,z_n),
\end{align}
for every $f\in\mathcal{A}(\D)$, where $\Omega=\cir$.
\item[(ii)] Let $V_{j}\in\mathcal{L}_{2}(\Mcal,\tr)$ for $1\le j\le n$. If $\{\A,\B\}\in\mathcal{N}_{n}$, then there exist finite measures $\mu_{ij}$, $1\leq i\leq j\leq n$, on $\Omega^n$ such that
\begin{align}
\label{Anna2eq1}&\tr\left(f(\B)-f(\A)-\frac{d}{ds}\bigg|_{s=0}f(\X(s))\right)=\sum_{1\leq j\leq n}\int_{\Omega^n}\frac{\partial^2 f }{\partial z_j^2}(z_1,\ldots,z_n)\,d\mu_{jj}(z_1,\ldots,z_n)\\
\nonumber&\hspace{2.5in}+2\sum_{1\leq i< j\leq n}\int_{\Omega^n}\frac{\partial^2 f }{\partial z_i\partial z_j}(z_1,\ldots,z_n)\,d\mu_{ij}(z_1,\ldots,z_n),
\end{align}
for every $f\in\mathcal{A}(\D)$, where $\Omega=\cir$.
\end{itemize}

If, in addition, $\A$ and $\B$ are tuples of self-adjoint contractions, then there exist  real-valued measures $\mu_j, \mu_{ij}$ for $1\le i\le j\le n$, such that \eqref{Anna1eq1} and \eqref{Anna2eq1} hold respectively with $\Omega=[-1,1]$ for $f\in\mathcal{A}(\Ra)$, $\epsilon>0$.
\end{thm}

It is worth mentioning that, in \cite{CGP23} authors obtained the trace formulas \eqref{Anna1eq1} and \eqref{Anna2eq1} corresponding to non-analytic multivariable scalar function class for various pairs of tuples of operators.

We would also like to highlight that in the single-variable setting for pairs of bounded operators $(T,V)$ with perturbation $V\in\hils$, authors of \cite{DySk09} and \cite{PoSkSu14} gave the existence of higher-order spectral shift measures $\mu_{n,T,V}$ such that
\begin{align}
\Tr\left(\mathscr{R}_{n,T,f}(V)\right)=\int_{\Omega}f^{(n)}(z)d\mu_{n,T,V}(z)
\end{align}
for sufficiently nice function $f$, where $\Omega=\R$ or $\cir$.

Motivated from \cite{DySk09,PoSkSu14} and by maintaining the harmony of Hilbert-Schmidt perturbations for single variable setting, in this article, we are interested to study the behavior of the higher-order Taylor remainders associated with multivariate operator functions under the same hypothesis as in \cite{Sk15}. More precisely, the following question is of particular interest.
\begin{prb}
Does there exist higher-order spectral shift measure for a pair of tuples of commuting contractions $(\A,\B)$ in $\Mcal$ such that $B_i-A_i\in\mLL(\Mcal,\tr),$ corresponding to the multivariate analytic function class $\mathcal{A}(\D)$?
\end{prb}
Below we briefly summarize our main results.
\vspace*{0.2cm}

\noindent\textbf{New results:}  
Note that in the single variable case, to obtain the higher-order spectral shift measures, one needs to estimate the trace of the higher-order derivative of an operator function in terms of the sup-norm of the same order derivative of the associated scalar function. Therefore, we should focus primarily on obtaining the explicit expression of the higher-order derivative of the function $t\mapsto f(\X(t))$ along with the estimation of its trace. Note that the operator derivatives $\dmds\bigg|_{s=t}f(\X(s))$ considered in \cite[Lemma 3.2, Lemma 4.4]{Sk15} are evaluated in $\|\cdot\|_{1}$ for $m=1$ and $ 2$ by assuming $V_{j},~1\le j\le n$, in $\mL(\Mcal, \tr)$ and $\mLL(\Mcal, \tr)$ respectively. Furthermore, in \cite{KsPoSlSu12}, it is evaluated in the Schatten p-norm $\|\cdot\|_{p}$ only for $m=1$ by considering $V_{j}\in\bph,\,\,1<p<\infty,$ along paths of tuples of bounded commuting self-adjoints $\X(s)$ with tangent vectors in the closure of the narrow tangent space
$$\Gamma_p^0(\A)=\left\{(\ii[A_1,Y]+Z_1,\ldots,\ii[A_n,Y]+Z_n):\, Y\in \bph,\, Z=(Z_{1},\ldots,Z_{n})\in(\{\A\}''\cap\bph)^{(n)}\right\},$$
where $\{\A\}''$ is a bicommutant of the family $\{A_1,\ldots,A_n\}$. For our purpose, these results are not sufficient to complete our aim. By considering linear path, and analytic functions $f\in\mathcal{A}(\D)$, we provide an explicit expression of  $\dmds\bigg|_{s=t}f(\X(s))$, $m\in \N$, for more general $\V$ and contractions $\X$ (see Lemma \ref{Lem2}). It is important to note that, in single variable case, 
\begin{equation}\label{TReq}
\Tr\left[\dmds\bigg|_{s=t} f(A+sV)\right]=\Tr\left[\dmmds\bigg|_{s=t} f'(A+sV)\cdot V\right],
\end{equation}
where $A\in\bh$, $V\in\mathcal{B}_{m}(\hil)$, and $f$ is a polynomial (see \cite[Lemma 2.2]{PoSkSu14}). However, in the case of multivariate functions, the above analogy \eqref{TReq} does not hold (see \cite[Equations 5.7 and 5.23]{CGP23}). In other words, the derivative of a multivariate operator function is more complex than the derivative of a single variable operator function when the initial operator and the perturbed operator do not commute and its complexity increases with the order of the differentiation.

Consider $m\ge 2,~1\le k\le m,~1\le j_{1}<\cdots<j_{k}\le n$ with $1\le i_{1},\ldots,i_{k}\le m$ such that $\sum_{l=1}^{k}i_{l}=m$ and $V_{j}\in\mLL(\Mcal, \tr)\,$ for $1\le j\le n$. Assume that either $m=2$ or $\tr$ is the standard trace. Suppose that there exists $t\in[0, 1]$ such that $\X(t)\in\cn$ satisfies \eqref{Dilation}, then we prove the following estimate
\begin{equation}\label{mainestimate}
\left|\tr\left(D_{f}^{j_{1}^{i_{1}},\ldots,j_{k}^{i_{k}}}(t)\right)\right|\le\left(\prod_{l=1}^{k}\|V_{j_{l}}\|_{2}^{i_{l}}\right)\left\|\frac{\partial^{m}f}{\partial z_{j_{1}}^{i_{1}}\cdots\partial z_{j_{k}}^{i_{k}}}\right\|_{L^{\infty}(\Omega^{n})},
\end{equation}
for every $f\in\mathcal{A}(\D)$,~$\epsilon>0$. In addition, if $A_{j}=A_{j}^{*},~ B_{j}=B_{j}^{*},~1\le j\le n$, and suppose there exists $t\in[0,1]$ such that $\X(t)\in\cn$, then the same holds for $f\in\mathcal{A}(\Ra)$,~$\epsilon>0$ (see Theorem \ref{Thm1}). {\sf Our proof involves tricky combinatorial approaches due to the increasing order of the differentiation of the function $t\mapsto f(\X(t))$, and surprisingly, after intricate calculations, we are able to get our above-desired estimates.}

Let $m\ge 2,~1\le k\le m$ and $V_{j}\in\mLL(\Mcal,\tr),$ for $1\le j\le n$. Assume that either $\tr$ is the standard trace or $m=2$. If $\{\A,\B\}\in\mathcal{N}_{n}$, then using our above estimate \eqref{mainestimate} we prove the existence of finite measures $\mu_{j_{1}^{i_{1}}\cdots j_{k}^{i_{k}}}$ for $1\le j_1<\cdots< j_k \le n$ and $1\le i_{1},\ldots,i_{k}\le m$ with $\sum_{l=1}^{k}i_{l}=m$ such that
\begin{align}
\label{mainformula}&\tr\left[f(\bm{B}_n)-\sum_{k=0}^{m-1}\frac{1}{k!}\,\frac{d^k}{ds^k}\bigg|_{s=0}f(\bm X_n(s))\right]=\\
\nonumber&\sum_{k=1}^{m}\left[\sum_{1\le j_{1}<\cdots<j_{k}\le n}^{}\left(\sum_{\substack{i_{1},\ldots,i_{k}\ge 1; \\ i_1+\cdots+i_{k}=m}}\frac{m!}{i_{1}!\cdots i_{k}!}\left(\int_{\Omega^{n}}\frac{\partial^{m}f}{\partial z_{j_{1}}^{i_{1}}\cdots\partial z_{j_{k}}^{i_{k}}}(z_1,\ldots,z_n)\,\,d\mu_{j_{1}^{i_{1}}\ldots j_{k}^{i_{k}}}(z_1,\ldots,z_n)\right)\right)\right],
\end{align}
for every $f\in\mathcal{A}(\D)$ (see Theorem \ref{Thm2}). In addition, if $\A$ and $\B$ are consist of self-adjoint operators, then there exist finite real-valued measures $\mu_{j_{1}^{i_{1}}\cdots j_{k}^{i_{k}}}$ such that the former holds with the integrals are evaluated over $[-1, 1]^{n}$. {\sf Therefore in the context of the existence of higher-order spectral shift measures, the above trace formula \eqref{mainformula} generalizes the trace formulas obtained in \cite[Theorem 5.1]{DySk09} and \cite[Theorem 1.3]{PoSkSu14} from a single-variable to a multivariable operator function case under Hilbert-Schmidt perturbations.}
\vspace*{0.2cm}

\noindent\textbf{Outline:} Apart from the Introduction, the article consists of three sections. In Section \ref{Sec2}, we recall some preliminaries on von Neumann inequality and multivariate normal dilations. In Section \ref{Sec3}, we discuss some facts related to divided differences and establish the explicit expression of higher-order derivatives of the function $t\mapsto f(\X(t)).$ Finally, in Section \ref{Sec4}, we prove the existence of higher-order spectral shift measures by estimating the trace of higher-order derivatives.

\section{Preliminaries}\label{Sec2}
In this section, we recall definitions and some preliminaries on von Neumann inequality and multivariate commuting normal dilation. 

Recall that, for  a tuple of commuting bounded  normal operators $\X$ and given a bounded Borel function $f$ on $\sigma(\X)$ (the joint spectrum of $\X=(X_{1},\ldots,X_{n})$), the operator function $f(\X)$ is representable by the integral 
\begin{align}
\label{Spectral}&f(\X)\overset{\text{def}}=\underbrace{\idotsint}\limits_{\sigma(\X)} f(\lambda_{1},\ldots,\lambda_{n})\,dE(\lambda_{1},\ldots,\lambda_{n}),
\end{align}
where $E$ is the product of the spectral measures $E_{j}$ of the operators $X_{j},\,1\le j\le n$ {\normalfont\cite[Theorem 6.5.1]{BSbook}, and it is supported on $\sigma(\X)$. It follows from \eqref{Spectral} that  $\X$ satisfies the von Neumann inequality,
\begin{equation}\label{vonNeumann}
\|f(\X)\|\le\|f\|_{L^{\infty}(\cir^{n})},\end{equation}
for $f\in\mathcal{A}(\D),\,\,\epsilon>0.$ It is well-known that the von Neumann inequality holds for all single contractions \cite{Neumann51} and for all pairs of commuting contractions \cite{Ando63}. A sufficient condition for an $n$-tuple of commuting contractions satisfying \eqref{vonNeumann} is established in \cite{GVVW09}.
	
The results of our study will be demonstrated for tuples of commuting contractions that can be dilated to tuples of commuting normal contractions. That is, we will consider those $\bm{X}_{n}\in\mathcal{C}_{n}$ for which there exists a Hilbert space $\mK\supset\hil$ and a tuple of commuting normal contractions $(U_{1},\ldots,U_{n})$ on $\mK$ such that
\begin{equation}\label{Dilation}
X_{j_{1}}\ldots X_{j_{l}}=P_{\hil}U_{j_{1}}\ldots U_{j_{l}}|_{\hil}^{},\quad l\in\N,\quad1\le j_{1},\ldots,j_{l}\le n,
\end{equation}
where $P_{\hil}$ is the orthogonal projection from $\mK$ onto $\hil$. It should be noted that, if $\X\in\cn$ satisfies \eqref{Dilation}, then it also  satisfies the von Neumann inequality \eqref{vonNeumann}. The unitary dilation of a single contraction \cite{Nagy53} as well as the commuting unitary dilation for pair of commuting contractions \cite{Ando63} are well known. The \cite[Corollary 4.9]{Pisier} gives a necessary and sufficient condition for the existence of multivariate commuting unitary dilation \eqref{Dilation}.

\section{Divided difference and technical results}\label{Sec3}
In this section, we provide some facts on divided differences to be used in a later section. In the evaluation of directional derivatives of single-variable operator functions, divided difference plays a major role. For multivariate operator functions, we require a modified representation of the divided difference.

Let $f\in C^{k}(\C)$ and $\lambda_{0},\ldots,\lambda_{k}\in\C$. We recall that the divided difference of zeroth order is the function itself, i.e., $f[\lambda_{0}]=f(\lambda_{0}).$ Then the divided difference of order $k$ is defined recursively by 
\begin{align}
\label{Divdiff}f[\lambda_{0},\ldots,\lambda_{k-2},\lambda_{k-1},\lambda_{k}]=\lim_{\lambda\to\lambda_{k}}\frac{f[\lambda_{0},\ldots,\lambda_{k-2},\lambda]-f[\lambda_{0},\ldots,\lambda_{k-2},\lambda_{k-1}]}{\lambda-\lambda_{k-1}}.
\end{align}
Let $\phi\in\mathcal{A}(\Cn).$ We define the $i_{1}$-th order divided difference of the function $\phi$ at $j_{1}$-th coordinate by
\begin{equation}\label{Divdiff2}
D_{j_{1}^{i_{1}}}^{\lambda_{1,0},\ldots,\lambda_{1,i_{1}}}\phi(z_{1},\ldots,z_{n})=\phi(z_1,\ldots,z_{j_{1}-1},\left[\lambda_{1,0},\ldots,\lambda_{1,i_{1}}\right],z_{j_{1}+1},\ldots,z_{n}),
\end{equation}
where $\lambda_{1,0},\ldots,\lambda_{1,i_{1}}\in\C.$ Similarly, for $\phi\in\mathcal{A}(\Cn)$ and $\lambda_{1,0},\ldots,\lambda_{1,i_{1}},\lambda_{2,0},\ldots,\lambda_{2,i_{2}}\in\C$, we define the $i_{1}$-th order divided difference at $j_{1}$-th coordinate and the $i_{2}$-th order divided difference at $j_{2}$-th coordinate of the function $\phi$ by
\begin{align}
\nonumber&D_{j_{1}^{i_{1}},j_{2}^{i_{2}}}^{\lambda_{1,0},\ldots,\lambda_{1,i_{1}},\lambda_{2,0},\ldots,\lambda_{2,i_{2}}}\phi(z_{1},\ldots,z_{n})=D_{j_{1}^{i_{1}}}^{\lambda_{1,0},\ldots,\lambda_{1,i_{1}}}D_{j_{2}^{i_{2}}}^{\lambda_{2,0},\ldots,\lambda_{2,i_{2}}}\phi(z_{1},\ldots,z_{n})\\
\label{Divdiff3}&\hspace*{0.7in}=\phi(z_1,\ldots,z_{j_{1}-1},\left[\lambda_{1,0},\ldots,\lambda_{1,i_{1}}\right],z_{j_{1}+1},\ldots,z_{j_{2}-1},\left[\lambda_{2,0},\ldots,\lambda_{2,i_{2}}\right],z_{j_{2}+1},\ldots,z_{n}).
\end{align}
Note that,
\begin{align}
\label{Divdiff4}&D_{j_{1}^{i_{1}}}^{\lambda_{1,0},\ldots,\lambda_{1,i_{1}}}D_{j_{2}^{i_{2}}}^{\lambda_{2,0},\ldots,\lambda_{2,i_{2}}}\phi(z_{1},\ldots,z_{n})=D_{j_{2}^{i_{2}}}^{\lambda_{2,0},\ldots,\lambda_{2,i_{2}}}D_{j_{1}^{i_{1}}}^{\lambda_{1,0},\ldots,\lambda_{1,i_{1}}}\phi(z_{1},\ldots,z_{n}).
\end{align}
The next two lemmas are crucial to obtain our main estimate and they generalize  \cite[Lemmas 4.1, 4.2]{Sk15}.
\begin{lma}\label{Lem5}
Let $\lambda_{l,0},\ldots,\lambda_{l,i_{l}}\in\C$ and $t_{l,1},\ldots,t_{l,i_{l}}\in\R$, for $1\le l\le k$ and $1\le k\le m.$ If $\phi\in\mathcal{A}(\Cn)$ and  $z_{1},\ldots,z_{n}\in\C$, then
\begin{align}
\nonumber&\left(\prod_{l=1}^{k}D_{j_{l}^{i_{l}}}^{\lambda_{l,0},\ldots,\lambda_{l,i_{l}}}\right)\phi(z_{1},\ldots,z_{n})=\prod_{r=1}^{k}\left[\int_{0}^{1}dt_{r,1}\int_{0}^{t_{r,1}}dt_{r,2}\cdots\int_{0}^{t_{r,i_{r}-1}}dt_{r,i_{r}}\right]\\
\nonumber&\frac{\partial^{m}}{\partial z_{j_{1}}^{i_{1}}\cdots\partial z_{j_{k}}^{i_{k}}}\,\phi\Bigg(z_{1},\ldots,z_{j_{1}-1},\left(\sum_{l=1}^{i_{1}}(\lambda_{1,l-1}-\lambda_{1,l})t_{1,i_{1}+1-l}+\lambda_{1,i_{1}}\right),z_{j_{1}+1},\ldots,z_{j_{2}-1},\\
\label{Lem5r1}&\left(\sum_{l=1}^{i_{2}}(\lambda_{2,l-1}-\lambda_{2,l})t_{2,i_{2}+1-l}+\lambda_{2,i_{2}}\right),z_{j_{2}+1},\ldots,z_{j_{k}-1},\left(\sum_{l=1}^{i_{k}}(\lambda_{k,l-1}-\lambda_{k,l})t_{k,i_{k}+1-l}+\lambda_{k,i_{k}}\right),\ldots,z_{n}\Bigg),
\end{align}
and 
\begin{align}
\label{Lem5r2}&\sup_{\substack{z_{1},\ldots,z_{n},\lambda_{l,0},\ldots,\lambda_{l,i_{l}}\in\C;\\1\le l\le k}}~\left|\left(\prod_{l=1}^{k}D_{j_{l}^{i_{l}}}^{\lambda_{l,0},\ldots,\lambda_{l,i_{l}}}\right)\phi(z_{1},\ldots,z_{n})\right|\le\frac{1}{i_{1}!\cdots i_{k}!}\left\|\frac{\partial^{m}\phi}{\partial z_{j_{1}}^{i_{1}}\cdots\partial z_{j_{k}}^{i_{k}}}\right\|_{\infty},
\end{align}
where $1\le i_{1},\ldots,i_{k}\le m$ and $\sum_{l=1}^{k}i_{l}=m.$
\end{lma}
\begin{proof}
Note that \eqref{Lem5r2} is an immediate consequence of the representation \eqref{Lem5r1}.\\
The proof of \eqref{Lem5r1} goes by induction on $k$. If $k=1$, then by applying \cite[Lemma 4.1]{Sk15} to the function $\phi\in\mathcal{A}(\Cn)$ we obtain
\begin{align}
\nonumber&D_{j_{q}^{i_{q}}}^{\lambda_{q,0},\ldots,\lambda_{q,i_{q}}}\phi(z_{1},\ldots,z_{n})=\int_{0}^{1}dt_{q,1}\int_{0}^{t_{q,1}}dt_{q,2}\cdots\int_{0}^{t_{q,i_{q}-1}}dt_{q,i_{q}}\\
\label{Lem5r3}&\hspace*{0.7in}\frac{\partial^{i_{q}}}{\partial z_{j_{q}}^{i_{q}}}\,\phi\left(z_{1},\ldots,z_{j_{q}-1},\left(\sum_{l=1}^{i_{q}}(\lambda_{q,l-1}-\lambda_{q,l})t_{q,i_{q}+1-l}+\lambda_{q,i_{q}}\right),z_{j_{q}+1},\ldots,z_{n}\right).
\end{align}
Let the formula be true for $k=q-1~(2\le q\le m)$. From \eqref{Divdiff4} we derive
\begin{align}
\nonumber&\prod_{l=1}^{q}\left(D_{j_{l}^{i_{l}}}^{\lambda_{l,0},\ldots,\lambda_{l,i_{l}}}\right)\phi(z_{1},\ldots,z_{n})\\
\label{Lem5r5}&=D_{j_{q}^{i_{q}}}^{\lambda_{q,0},\ldots,\lambda_{q,i_{q}}}\left[\prod_{l=1}^{q-1}\left(D_{j_{l}^{i_{l}}}^{\lambda_{l,0},\ldots,\lambda_{l,i_{l}}}\right)\phi(z_{1},\ldots,z_{n})\right].
\end{align}
Now the result trivially holds if $\lambda_{q,0},\ldots,\lambda_{q,i_{q}}$ are all distinct in $\C$. If not, then due to the analyticity of $\phi\in\mathcal{A}(\Cn)$ \eqref{Lem5r5} reduces to
\begin{align}
\nonumber&\prod_{r=1}^{q-1}\left[\int_{0}^{1}dt_{r,1}\int_{0}^{t_{r,1}}dt_{r,2}\cdots\int_{0}^{t_{r,i_{r}-1}}dt_{r,i_{r}}\right]\\
\nonumber&D_{j_{q}^{i_{q}}}^{\lambda_{q,0},\ldots,\lambda_{q,i_{q}}}\frac{\partial^{i_{1}+\cdots+i_{q-1}}}{\partial z_{j_{1}}^{i_{1}}\cdots\partial z_{j_{q-1}}^{i_{q-1}}}\,\phi\Bigg(z_{1},\ldots,z_{j_{1}-1},\left(\sum_{l=1}^{i_{1}}(\lambda_{1,l-1}-\lambda_{1,l})t_{1,i_{1}+1-l}+\lambda_{1,i_{1}}\right),z_{j_{1}+1},\ldots,\\
\nonumber&\hspace{1.2in}z_{j_{2}-1},\left(\sum_{l=1}^{i_{2}}(\lambda_{2,l-1}-\lambda_{2,l})t_{2,i_{2}+1-l}+\lambda_{2,i_{2}}\right),z_{j_{2}+1},\ldots,\\
\label{Lem5r6}&\hspace{1.2in}z_{j_{q-1}-1},\left(\sum_{l=1}^{i_{q-1}}(\lambda_{q-1,l-1}-\lambda_{q-1,l})t_{q-1,i_{q-1}+1-l}+\lambda_{q-1,i_{q-1}}\right),z_{j_{q-1}+1},\ldots,z_{n}\Bigg),
\end{align}
which concludes \eqref{Lem5r1}.
\end{proof}
\begin{lma}\label{Lem6}
Let $N\in\N$, $\lambda_{i}\in\C$, for $1\le i\le m$ and
\begin{align}
\label{Lem6r1}&f_{N}(z_1,\ldots,z_n)=\sum_{0\le k_1,\ldots,k_n\le N}^{}c_{k_1,\ldots,k_n}z_{1}^{k_{1}}\cdots z_{n}^{k_{n}}.
\end{align}
Then the following assertions hold. 
\begin{itemize}
\item[(i)]\label{Lem6i} If $k_{j}\ge m$ and $h\neq 0$, then
\begin{align}
\nonumber&f_{N}(z_{1},\ldots,z_{j-1},[\lambda_{1},\ldots,\lambda_{m-1},\lambda_{m}+h,\lambda_{m}+h],z_{j+1},\ldots,z_{n})\\
\label{Lem6r2}&\quad=\sum_{0\le k_1,\ldots,k_n\le N}^{}c_{k_1,\ldots,k_n}\sum_{\substack {p_{0},\ldots,p_{m}\ge 0;\\ p_{0}+\cdots+p_{m}=k_{j}-m}}z_{1}^{k_{1}}\cdots z_{j-1}^{k_{j-1}}(\lambda_{m}+h)^{p_{0}+p_{m}}\left(\prod_{l=1}^{m-1}(\lambda_{l})^{p_{l}}\right)z_{j+1}^{k_{j+1}}\cdots z_{n}^{k_{n}}.
\end{align}
\item[(ii)]\label{Lem6ii} If $k_{j_{l}}\ge i_{l}$ for $1\le l\le k,$ then
\begin{align}
\label{Lem6r4}&\left(\prod_{l=1}^{k}D_{j_{l}^{i_{l}}}^{\lambda_{l,0},\ldots,\lambda_{l,i_{l}}}\right)f_{N}(z_{1},\ldots,z_{n})=\left(\prod_{l=1}^{k}\sum_{\substack {p_{l,0},\ldots,p_{l,i_{l}}\ge 0;\\p_{l,0}+\cdots+ p_{l,i_{l}}=k_{j_{l}}-i_{l}}}\right)\left[\left(\prod_{\substack{m=1\\m\notin\{j_{1},\ldots,j_{k}\}}}^{n}z_{m}^{k_{m}}\right)\prod_{r=1}^{k}\prod_{s=0}^{i_{r}}(\lambda_{r,s})^{p_{r,s}}\right].
\end{align}
\end{itemize}
\end{lma}
\begin{proof}
Because of the linearity of the divided difference, instead of \eqref{Lem6r1}, it will be sufficient to work with $P(z_{1},\ldots,z_{n})=z_{1}^{k_{1}}\cdots z_{n}^{k_{n}},$  where $k_{1},\ldots,k_{n}\in\N\cup\{0\}$. The following conclusion is straightforward.
\begin{align}
\nonumber P(z_{1},\ldots,z_{j-1},&[\lambda_{1},\ldots,\lambda_{m}],z_{j+1},\ldots,z_{n})\\
\label{Lem6r3}&=\sum_{\substack {p_{0},\ldots,p_{m-1}\ge 0;\\ p_{0}+\cdots+p_{m-1}=k_{j}-(m-1)}}z_{1}^{k_{1}}\cdots z_{j-1}^{k_{j-1}} \left(\prod_{l=1}^{m}(\lambda_{l})^{p_{l-1}}\right)z_{j+1}^{k_{j+1}}\cdots z_{n}^{k_{n}}.
\end{align}
On the other hand, applying \eqref{Divdiff} and using \eqref{Lem6r3}, we conclude
\begin{align}
\nonumber P(z_{1},\ldots,z_{j-1},[\lambda_{1}&,\ldots,\lambda_{m-1},\lambda_{m}+h,\lambda_{m}+h],z_{j+1},\ldots,z_{n})\\
\nonumber&=\ddt\bigg|_{t=\lambda_{m}+h}^{}P(z_{1},\ldots,z_{j-1},[\lambda_{1},\ldots,\lambda_{m-1},t],z_{j+1},\ldots,z_{n})\\
\nonumber&=\sum_{\substack {p_{0},\ldots,p_{m}\ge 0;\\ p_{0}+\cdots+p_{m}=k_{j}-m}}z_{1}^{k_{1}}\cdots z_{j-1}^{k_{j-1}}(\lambda_{m}+h)^{p_{0}+p_{m}}\left(\prod_{l=1}^{m-1}(\lambda_{l})^{p_{l}}\right)z_{j+1}^{k_{j+1}}\cdots z_{n}^{k_{n}},
\end{align}
which proves \eqref{Lem6r2}. Subsequent evaluations of the divided differences to \eqref{Lem6r3} prove \eqref{Lem6r4}.
\end{proof}

The following differentiation formula was stated in \cite[Lemma 2.1]{PoSkSu14}, where the derivatives exist in the uniform operator topology, but here we provide a short proof of the existence of the derivatives in Schatten norms, which is necessary in our context. 
\begin{lma}\label{Lem1}
Let $H$ and $V$ be two bounded operators and let $p$ be a natural number. Then
\begin{align}
&\label{Lem1r1}(H+V)^p-{H}^p=\sum_{\substack {p_0,\,p_1\ge 0;\\p_0+p_1=p-1}}(H+V)^{p_0}V{H}^{p_1},
\end{align}  
and for $n\le p$,\\
\begin{align}
&\label{Lem1r2}\frac{d^n}{ds^n}\bigg|_{s=t}^{}(H+sV)^p=n!\sum_{\substack {p_0,\ldots,\,p_n\ge 0;\\ p_0+\cdots+p_n=p-n}}\left[\prod_{i=0}^{n-1}\big((H+tV)^{p_i}\,V\big)(H+tV)^{p_n}\right],
\end{align}
where the derivative \eqref{Lem1r2} is evaluated in the operator norm. In addition, if $V\in\mathcal{L}_{r}(\Mcal,\tr)$, then \eqref{Lem1r2} exists in $\|\cdot\|_{\frac{r}{n}}$ for $r>n$ and for $r\le n$ it exists in $\|\cdot\|_{1}$.
 
\end{lma}
\begin{proof}
The representation \eqref{Lem1r1} is trivial. Assume $V\in\mathcal{L}_{r}(\Mcal,\tr)$ and $r> n$. First we provide a proof of the existence of \eqref{Lem1r2} in $\|\cdot\|_{\frac{r}{n}}$ for $r>n$.  The proof goes by induction on $n$. The cases $n=1$ and $n=2$ follow from {\cite[Lemma 3.1]{Sk15}}. Suppose \eqref{Lem1r2} is true for $n=k-1~(2\le k\le p)$, that is,
\begin{align}
\label{Lem1r3}\frac{d^{k-1}}{ds^{k-1}}\bigg|_{s=t}^{}(H+sV)^p=(k-1)!\sum_{\substack {p_0,\ldots,p_{k-1}\ge 0;\\ p_0+\cdots+p_{k-1}=p-(k-1)}}\left[\prod_{i=0}^{k-2}\big((H+tV)^{p_i}\,V\big)(H+tV)^{p_{k-1}}\right]
\end{align}
exists in $\|\cdot\|_{\frac{r}{k-1}}$. Now by
using \eqref{Lem1r1} we have
\begin{align}
\nonumber \varphi(\epsilon)&=\frac{1}{\epsilon}\left[\frac{d^{k-1}}{ds^{k-1}}\bigg|_{s=t+\epsilon}^{}(H+sV)^{p}-\frac{d^{k-1}}{ds^{k-1}}\bigg|_{s=t}^{}(H+sV)^p\right]\\
\nonumber&=\frac{(k-1)!}{\epsilon}\sum_{\substack {p_0,\ldots,\,p_{k-1}\ge 0;\\ p_0+\cdots+p_{k-1}=p-(k-1)}}\left[\sum_{i=0}^{k-1}g_{i,p_{0},\ldots,p_{k-1}}(\epsilon)\right]\\
\label{Lem1r4}&=(k-1)!\sum_{\substack {p_0,\ldots,\,p_{k}\ge 0;\\ p_0+\cdots+p_{k}=p-k}}\left[\sum_{i=0}^{k-1}h_{i,p_{0},\ldots,p_{k}}(\epsilon)\right],
\end{align}
where
\begin{align}
\nonumber&g_{i,p_{0},\ldots,p_{k-1}}(\epsilon)\\
\nonumber&=\prod_{\substack{j=0\\i>0}}^{i-1}\bigg((H+(t+\epsilon)V)^{p_{j}}V\bigg)\bigg((H+(t+\epsilon)V)^{p_{i}}-(H+tV)^{p_{i}}\bigg)\prod_{\substack{q=i+1\\i<k-1}}^{k-1}\bigg(V(H+tV)^{p_{q}}\bigg),
\end{align}
and
\begin{align}
\nonumber&h_{i,p_{0},\ldots,p_{k}}(\epsilon)=\prod_{j=0}^{i}\bigg((H+(t+\epsilon)V)^{p_{j}}V\bigg)(H+tV)^{p_{i+1}}\prod_{\substack{q=i+2\\i<k-1}}^{k}\bigg(V(H+tV)^{p_{q}}\bigg).
\end{align}
Again by using the representation \eqref{Lem1r1} we obtain

\begin{align}
\nonumber\psi(\epsilon)&=\varphi(\epsilon)-k!\sum_{\substack{p_0,\ldots,\,p_k\ge0;\\p_0+\cdots+p_k=p-k}}\left[\prod_{i=0}^{k-1}\big((H+tV)^{p_{i}}\,V\big)(H+tV)^{p_{k}}\right]\\
\nonumber&=\epsilon(k-1)!\sum_{\substack {p_0,\ldots,p_k\ge 0;\\p_0+\cdots+p_k=p-k}}\Bigg[\sum_{i=0}^{k-1}\,(k-i)\sum_{\substack {p_{i,0},\,p_{i,1}\ge 0;\\p_{i,0}+p_{i,1}=p_{i}-1}}~\prod_{\substack{j=0\\i>0}}^{i-1}\bigg((H+(t+\epsilon)V)^{p_{j}}V\bigg)\\
\label{Lem1r5}&\hspace*{1.5in}\times(H+(t+\epsilon )V)^{p_{i,0}}\,V(H+tV)^{p_{i,1}}\prod_{q=i+1}^{k}\bigg(V(H+tV)^{p_{q}}\bigg)\Bigg].
\end{align}
Therefore by the triangle inequality from  \eqref{Lem1r5} we get
\begin{align}
\nonumber\|\psi(\epsilon)\|_{\frac{r}{k}}&\le\epsilon(k-1)!\,\|V\|\,\|V\|_{r}^{k}\,K,
\end{align}
where $K$ is some constant and hence it completes the cycle of the induction on $n$. This concludes the proof for the existence of \eqref{Lem1r2} in $\|\cdot\|_{\frac{r}{n}}$. Similarly if $r\le n,$ then $V\in\mathcal{L}_{n}(\Mcal,\tr),$ and it immediately implies the existence of $\frac{d^{n-1}}{ds^{n-1}}\big|_{s=t}(H+sV)^{p}$ in $\|\cdot\|_{\frac{n}{n-1}}.$ This ensures that \eqref{Lem1r2} exists in $\|\cdot\|_{1}$.
\end{proof}	
To estimate the traces of higher-order derivatives of multivariate operator function $t\mapsto f(\X(t))$, it is necessary to obtain an explicit expression of the higher-order derivatives of the above operator function. The following lemma provides the required expression.
\begin{lma}\label{Lem2}
Assume Notations \ref{Notations} and let $f\in\mathcal{A}(\D)$. Then for $m\in\N,$ the G\^{a}teaux derivative $\dmds\bigg|_{s=t}^{}f(\bm{X}_n(s))$ exists and the map $[0,1]\ni t\mapsto \dmds\bigg|_{s=t}^{}f(\bm{X}_n(s))$ is continuous in the operator norm. Moreover, if $f$ is given by  \eqref{srf}, then\\
\begin{align}
\label{Lem2r1}&\dmds\bigg|_{s=t}^{}f(\bm{X}_n(s))=\sum_{k=1}^{m}\left[\sum_{1\le j_{1}<\cdots<j_{k}\le n}^{}\left(\sum_{\substack{i_{1},\ldots,i_{k}\ge 1;\\i_1+\cdots+i_{k}=m}}\left(\frac{m!}{i_{1}!\cdots i_{k}!}\,\,D_{f}^{j_{1}^{i_1},\ldots,j_{k}^{i_{k}}}(t)\right)\right)\right],
\end{align}
where
\begin{align}
\nonumber&D_{f}^{j_{1}^{i_1},\ldots,j_{k}^{i_{k}}}(t)\\
\nonumber&=\sum_{k_1,\ldots,k_n\ge0}c_{k_1,\ldots,k_n}T_{k_1,\ldots,k_{j_{1}-1}}(\bm{X}_{j_{1}-1}(t)) \prod_{l=1}^{k-1}\left[\frac{d^{i_{l}}}{ds^{i_{l}}}\bigg|_{s=t}^{}(X_{j_{l}}(s))^{k_{j_{l}}}T_{k_{j_{l}+1},\ldots,k_{j_{l+1}-1}}({}_ {j_{l}+1}\bm{X}_{j_{l+1}-1}(t))\right]\\
\label{Lem2r10}&\hspace{2.8in} \times\frac{d^{i_{k}}}{ds^{i_{k}}}\bigg|_{s=t}^{}(X_{j_{k}}(s))^{k_{j_{k}}}T_{k_{j_{k}+1},\ldots,k_{n}}({}_ {j_{k}+1}\bm{X}_{n}(t)).
\end{align}

Furthermore, if $V_{j}=B_j-A_j\in\mathcal{L}_{r}(\Mcal,\tr)$, then  $t\mapsto\dmds\bigg|_{s=t}^{}f(\bm{X}_n(s))$ exists and is continuous on $[0, 1]$ in $\|\cdot\|_{\frac{r}{m}}$ for $r>m$ and in $\|\cdot\|_{1}$ for $r\le m$.	
\end{lma}
We would like to mention that a similar formula (like \eqref{Lem2r1}) for higher-order derivatives of a {\sf single variable} operator function in unitary path was obtained in \cite[Theorem 5.3.4]{Sk19}, but in our context, we are dealing with multivariate operator functions and the path we consider here is linear, and hence our formula \eqref{Lem2r1} is more complicated compared to the formula given in \cite[Theorem 5.3.4]{Sk19}. {\sf In other words, we need more intricate combinatorics to get our formula, which is one of the major difficulties we face in this direction. Since the proof of this lemma is very technical, hence for readers' convenience, we discuss the proof in detail.}\vspace*{-0.2cm}
\begin{proof}[Proof of Lemma \ref{Lem2}]
We only establish the existence of $\dmds\bigg|_{s=t}^{}f(\bm{X}_n(s))$ in the respective norms; and the continuity of the map $[0,1]\ni t\mapsto\dmds\bigg|_{s=t}^{}f(\bm{X}_n(s))$ can be proved completely analogously. Let $f\in\mathcal{A}(\D)$ be given by \eqref{srf}. Our proof consists of two parts: Firstly, we prove the formula \eqref{Lem2r1}, and then by showing the existence of $\dmds\bigg|_{s=t}^{}f(\bm{X}_n(s))$ in the respective norms, we conclude the proof. Denote
\begin{align}
\nonumber g_{j_{1}^{i_1},\ldots,j_{k}^{i_{k}}}(t)=&T_{k_1,\ldots,k_{j_{1}-1}}(\bm{X}_{j_{1}-1}(t))\prod_{q=1}^{k-1}\left[\prod_{l=0}^{i_{q}-1}\big((X_{j_{q}}(t))^{p_{q,l}}V_{j_{q}}\big)T_{p_{q,i_{q}},k_{j_{q}+1},\ldots,k_{j_{q+1}-1}}({}_ {j_{q}}\bm{X}_{j_{q+1}-1}(t))\right]\\
\nonumber&\hspace{1.57in} \times\prod_{l=0}^{i_{k}-1}\big((X_{j_{k}}(t))^{p_{k,l}}V_{j_{k}}\big)T_{p_{k,i_{k}},k_{j_{k}+1},\ldots,k_{n}}({}_ {j_{k}}\bm{X}_{n}(t)),
\end{align}
where $\sum_{l=1}^{k}i_{l}=m$.
From \eqref{Lem1r2} it is straightforward to see that
\begin{align}
\label{Lem2r3}&D_{f}^{j_{1}^{i_1},\ldots,j_{k}^{i_{k}}}(t)=\sum_{k_1,\ldots,k_n\ge0}c_{k_1,\ldots,k_n}\left[\,\left(\prod_{l=1}^{k}i_{l}!\sum_{\substack {p_{l,0},\ldots,p_{l,i_{l}}\ge 0;\\p_{l,0}+\cdots+ p_{l,i_{l}}=k_{j_{l}}-i_{l}}}\right)g_{j_{1}^{i_1},\ldots,j_{k}^{i_{k}}}(t)\right],
\end{align}
which implies
\begin{align}
\label{Lem2r4}&\sup_{t\,\in\,[0, 1]}\left\|D_{f}^{j_{1}^{i_1},\ldots,j_{k}^{i_{k}}}(t)\right\|\le\left(\prod_{l=1}^{k}\|V_{j_{l}}\|^{i_{l}}\right)\sum_{k_1,\ldots,k_n\ge0}\mathscr{K}_{k_{1},\ldots,k_{n}}\,|c_{k_1,\ldots,k_n}|,
\end{align}
where $\mathscr{K}_{k_{1},\ldots,k_{n}}$ is the co-efficient of \,$\frac{\partial^{m}}{\partial z_{j_{1}}^{i_{1}}\cdots \partial z_{j_{k}}^{i_{k}}}\left(z_{1}^{k_{1}}\cdots z_{n}^{k_{n}}\right)$. Note that the series in the right-hand side of \eqref{Lem2r4} converges since all the partial derivatives of \eqref{srf} converge absolutely, and therefore, the right-hand side of \eqref{Lem2r1} makes sense. Furthermore, the sum $\sum_{p_{l,0}+\cdots+ p_{l,i_{l}}=k_{j_{l}}-i_{l}}^{}$ is empty for $k_{j_{l}}< i_{l}$.

We use induction on $m$ to prove the formula \eqref{Lem2r1}. The cases $m=1$ and $2$ follow from {\cite[Lemmas 3.2, 4.4]{Sk15}}. We assume that \eqref{Lem2r1} holds for $m=q-1~(q\ge 2)$, that is,
\begin{equation}\label{Lem2r2}
\frac{d^{q-1}}{ds^{q-1}}\bigg|_{s=t}^{}f(\bm{X}_n(s))=\sum_{k=1}^{q-1}\left[\sum_{1\le j_{1}<\cdots<j_{k}\le n}^{}\left(\sum_{\substack{i_{1},\ldots,i_{k}\ge 1 ;\\ i_1+\cdots+i_{k}=q-1}}\left(\frac{(q-1)!}{i_{1}!\cdots i_{k}!}\,\,D_{f}^{j_{1}^{i_1},\ldots,j_{k}^{i_{k}}}(t)\right)\right)\right]
\end{equation}
and verify below that it holds for $m=q$. Now to prove the formula \eqref{Lem2r1} is also true for $m=q$, we  need to show that after taking the derivative of \eqref{Lem2r2} the sum in the right-hand side should coincide with the sum in \eqref{Lem2r1} for $m=q$. To that aim, after performing the differentiation of \eqref{Lem2r2}, we need to analyze the term-by-term expressions of the derivative at each step $k\in\{1,\ldots,q-1\}$.

The derivative of the term for $k=1$ in \eqref{Lem2r2} is given by
\begin{align}
\nonumber\frac{d}{dt}\left(\sum_{1\le j_{1}\le n}^{}D_{f}^{j_{1}^{q-1}}(t)\right)&=\sum_{1\le j_{1}\le n}^{}D_{f}^{j_{1}^{q}}(t)+\sum_{1\le j_{1}<j_{2}\le n}^{}D_{f}^{j_{1}^{q-1},j_{2}}(t)+\sum_{1\le j_{2}<j_{1}\le n}^{}D_{f}^{j_{2},j_{1}^{q-1}}(t)\\
\nonumber&=\sum_{1\le j_{1}\le n}^{}D_{f}^{j_{1}^{q}}(t)+\sum_{1\le j_{1}<j_{2}\le n}^{}\left(D_{f}^{j_{1}^{q-1},j_{2}}(t)+D_{f}^{j_{1},j_{2}^{q-1}}(t)\right)\\
\label{Lem2r5}&=:\mathcal{S}_{1}+\mathcal{S}'_{1},
\end{align}
where $\mathcal{S}_{1}$ is the exact term for $k=1$ in \eqref{Lem2r1} for $m=q$. In addition, we have an extra term $\mathcal{S}'_{1}$ with it, this will help us in the next step to get the exact term for $k=2$ in \eqref{Lem2r1} for $m=q$.

The derivative of the term for $k=2$ in \eqref{Lem2r2} is given by
\begin{align}
\nonumber&\frac{d}{dt}\left[\sum_{1\le j_{1}<j_{2}\le n}^{}\left(\sum_{\substack{i_{1},i_{2}\ge 1;\\ i_1+i_{2}=q-1}}\left(\frac{(q-1)!}{i_{1}!\, i_{2}!}\,\,D_{f}^{j_{1}^{i_1},j_{2}^{i_{2}}}(t)\right)\right)\right]\\
\nonumber&=\sum_{1\le j_{1}<j_{2}\le n}^{}\left[\,\underbrace{\sum_{\substack{i_{1},i_{2}\ge 1;\\ i_1+i_{2}=q-1}}\frac{(q-1)!}{i_{1}!\,i_{2}!}\,D_{f}^{j_{1}^{i_{1}+1},j_{2}^{i_{2}}}(t)}_{(i)}+\underbrace{\sum_{\substack{i_{1},i_{2}\ge 1;\\ i_1+i_{2}=q-1}}\frac{(q-1)!}{i_{1}!\,i_{2}!}\,D_{f}^{j_{1}^{i_{1}},j_{2}^{i_{2}+1}}(t)}_{(ii)}\right]\\
\nonumber&\quad+\sum_{1\le j_{1}<j_{2}<j_{3}\le n}^{}\left[\,\sum_{\substack{i_{1},i_{2}\ge 1,i_{3}=1;\\ i_1+i_{2}=q-1}}\frac{(q-1)!}{i_{1}!\,i_{2}!}\left(D_{f}^{j_{1}^{i_{1}},j_{2}^{i_{2}},j_{3}^{i_{3}}}(t)+D_{f}^{j_{1}^{i_{1}},j_{2}^{i_{3}},j_{3}^{i_{2}}}(t)+D_{f}^{j_{1}^{i_{3}},j_{2}^{i_{1}},j_{3}^{i_{2}}}(t)\right)\right]\\
\label{Lem2r6}&=:\mathcal{S}_{2}+\mathcal{S}'_{2}.
\end{align}
Next by substituting $p_{1}=i_{1}+1, p_{2}=i_{2}$ in the first summand $(i)$ and $p_{1}=i_{1}, p_{2}=i_{2}+1$ in the second summand $(ii)$ of $\mathcal{S}_{2}$ we obtain
\begin{align}
\nonumber&\mathcal{S}_{2}=\sum_{1\le j_{1}<j_{2}\le n}^{}\left[\,\sum_{\substack{p_{1}\ge 2,p_{2}\ge 1;\\ p_1+p_{2}=q}}\frac{(q-1)!}{(p_{1}-1)!\,p_{2}!}\,D_{f}^{j_{1}^{p_{1}},j_{2}^{p_{2}}}(t)\right]+\sum_{1\le j_{1}<j_{2}\le n}^{}\left[\,\sum_{\substack{p_{1}\ge 1,p_{2}\ge 2;\\ p_1+p_{2}=q}}\frac{(q-1)!}{p_{1}!\,(p_{2}-1)!}\,D_{f}^{j_{1}^{p_{1}},j_{2}^{p_{2}}}(t)\right]\\
\nonumber&=\sum_{1\le j_{1}<j_{2}\le n}^{}\left[\,\sum_{\substack{p_{1}\ge 2,p_{2}\ge 1;\\ p_1+p_{2}=q}}\frac{p_{1}\,(q-1)!}{p_{1}!\,p_{2}!}\,D_{f}^{j_{1}^{p_{1}},j_{2}^{p_{2}}}(t)\right]+\sum_{1\le j_{1}<j_{2}\le n}^{}\left[\,\sum_{\substack{p_{1}\ge 1,p_{2}\ge 2;\\ p_1+p_{2}=q}}\frac{p_{2}\,(q-1)!}{p_{1}!\,p_{2}!}\,D_{f}^{j_{1}^{p_{1}},j_{2}^{p_{2}}}(t)\right]\\
\label{Lem2r7}&=\sum_{1\le j_{1}<j_{2}\le n}^{}\left[\,\sum_{\substack{i_{1}\ge 2,i_{2}\ge 1; \\ i_1+i_{2}=q}}\frac{i_{1}\,(q-1)!}{i_{1}!\,i_{2}!}\,D_{f}^{j_{1}^{i_{1}},j_{2}^{i_{2}}}(t)\right]+\sum_{1\le j_{1}<j_{2}\le n}^{}\left[\,\sum_{\substack{i_{1}\ge 1,i_{2}\ge 2; \\ i_1+i_{2}=q}}\frac{i_{2}\,(q-1)!}{i_{1}!\,i_{2}!}\,D_{f}^{j_{1}^{i_{1}},j_{2}^{i_{2}}}(t)\right].
\end{align}
Now by adding the term $\mathcal{S}'_{1}$ (obtained earlier for $k=1$) and $\mathcal{S}_{2}$ gives the exact term for $k=2$ in \eqref{Lem2r1} for $m=q$, and
the additional extra term $\mathcal{S}'_{2}$ will help us subsequently to get the next exact term for $k=3$ in \eqref{Lem2r1} for $m=q$.

Continuing in this way, after considering the differentiation of the term $k=l$, for $1\le l\le q-1$ in \eqref{Lem2r2}, we have
\begin{align}
\nonumber&\frac{d}{dt}\left[\sum_{1\le j_{1}<\cdots<j_{l}\le n}^{}\left(\sum_{\substack{i_{1},\ldots,i_{l}\ge 1;\\ i_1+\cdots+i_{l}=q-1}}\left(\frac{(q-1)!}{i_{1}!\ldots i_{l}!}\,\,D_{f}^{j_{1}^{i_1},\ldots,j_{l}^{i_{l}}}(t)\right)\right)\right]\\
\nonumber&\hspace*{1cm}=\sum_{1\le j_{1}<\cdots<j_{l}\le n}^{}\left[\sum_{\substack{i_{1},\ldots,i_{l}\ge 1;\\ i_1+\cdots+i_{l}=q-1}}\left(\frac{(q-1)!}{i_{1}!\ldots i_{l}!}\,\,\sum_{r=1}^{l}D_{f}^{j_{1}^{i_1},\ldots,j_{r}^{i_{r}+1},\ldots,j_{l}^{i_{l}}}(t)\right)\right]\\
\nonumber&\hspace*{1cm}\quad+\sum_{1\le j_{1}<\cdots<j_{l+1}\le n}^{}\left[\sum_{r=1}^{l+1}\left(\sum_{\substack{i_{1},\ldots,i_{r-1},i_{r+1},\ldots,i_{l+1}\ge 1;\\ i_{r}=1;\\ i_1+\cdots+i_{r}+\cdots+i_{l+1}=q}}\left(\frac{(q-1)!}{i_{1}!\cdots i_{l+1}!}\,\,D_{f}^{j_{1}^{i_1},\ldots,j_{r}^{i_{r}},\ldots,j_{l+1}^{i_{l+1}}}(t)\right)\right)\right]\\
\label{Lem2r8}&\hspace*{1cm}=:\mathcal{S}_{l}+\mathcal{S}'_{l}.
\end{align}
Again performing similar kind of substitutions as in \eqref{Lem2r6} to $\mathcal{S}_{l}$ we obtain
\begin{align}
\label{Lem2r9}&\mathcal{S}_{l}=\sum_{1\le j_{1}<\cdots<j_{l}\le n}^{}\left[\sum_{r=1}^{l}\,\,\left(\sum_{\substack{i_{1},\ldots,i_{r-1},i_{r+1},\ldots,i_{l}\ge 1;\\ i_{r}\ge 2;\\ i_1+\cdots+i_{r}+\cdots+i_{l}=q}}\left(\frac{i_{r}(q-1)!}{i_{1}!\ldots i_{l}!}\,D_{f}^{j_{1}^{i_1},\ldots,j_{r}^{i_{r}},\ldots,j_{l}^{i_{l}}}(t)\right)\right)\right],
\end{align}
and hence $\left(\mathcal{S}'_{l-1}+\mathcal{S}_{l}\right)$ provides the exact term for $k=l~(2\le l\le q-1)$ in \eqref{Lem2r2} for $m=q$.

Now to complete the induction cycle on $m$, we need to show that after performing the differentiation of the term $k=q-1$ in \eqref{Lem2r2}, the additional extra term $\mathcal{S}'_{q-1}$ is exactly the same term for $k=q$ in \eqref{Lem2r1} for $m=q$. Indeed, the derivative of the term $k=q-1$ in \eqref{Lem2r2} gives the additional term as,
\begin{align}
\nonumber&\mathcal{S}'_{q-1}=\sum_{1\le j_{1}<\cdots<j_{q}\le n}^{}\left[(q-1)!\sum_{r=1}^{q}\,\,D_{f}^{j_{1},\ldots,j_{r},\ldots,j_{q}}(t)\right]=\sum_{1\le j_{1}<\cdots<j_{q}\le n}^{}q!\left(D_{f}^{j_{1},\ldots,j_{q}}(t)\right),
\end{align}
which is same as the term for $k=q$ in the sum \eqref{Lem2r1} corresponding to $m=q$. This completes the induction on $m$. Therefore our expected formula for the derivative $\dmds\bigg|_{s=t}f(\bm{X}_{n}(s))$ in  \eqref{Lem2r1} holds true. Finally, the existence of $\dmds\bigg|_{s=t}f(\bm{X}_{n}(s))$ in the Schatten norm and in the operator norm follows from \eqref{Lem2r10} and Lemma \ref{Lem1}. This completes the proof. 
\end{proof}

\section{Principle estimate and spectral shift measures}\label{Sec4}
Our next theorem establishes the estimate of the trace of higher-order derivatives corresponding to multivariate operator function.
\begin{thm}\label{Thm1}
Assume Notations \ref{Notations} and \eqref{Lem2r10}. Let $m\ge 2$ and $V_{j}\in\mathcal{L}_{2}(\Mcal,\tr),$ for $1\le j\le n$. Assume that either $\tr$ is the standard trace or $m=2$. Suppose there exists $t\in[0, 1]$ such that $\bm{X}_{n}(t)\,\in\,\mathcal{C}_{n}$ satisfies \eqref{Dilation}. Then, for every $f\,\in\mathcal{A}(\D),\,\epsilon>0$,
\begin{equation}\label{Thm1r1}
\left|\tr\left(D_{f}^{j_{1}^{i_{1}},\ldots,j_{k}^{i_{k}}}(t)\right)\right|\le\left(\prod_{l=1}^{k}(\|V_{j_{l}}\|_{2}^{i_{l}})\right)\left\|\frac{\partial^{m}f}{\partial z_{j_{1}}^{i_{1}}\cdots\partial z_{j_{k}}^{i_{k}}}\right\|_{L^{\infty}(\Omega^{n})},
\end{equation}
where $1\le k\le m,~1\le j_1<\cdots< j_k\le n$ along with $1\le i_{1},\ldots,i_{k}\le m$ such that $\sum_{l=1}^{k} i_{l}=m$ and \,\,$\Omega\,=\,\cir$. If, in addition, $\bm{X}_{n}(t)$ is a tuple of self-adjoint contractions, then \eqref{Thm1r1} holds with $\Omega\,=\,[-1, 1],$ for every $f\in\mathcal{A}(\Ra),\,\epsilon>0$.
\end{thm}

The basic idea of the proof of the above Theorem \ref{Thm1} is based on a tricky adaptation of the approach considered in the proof of \cite[Theorem 4.6]{Sk15}. In the expression of higher-order derivatives, the increasing number of perturbation operators $V_{j}$ creates major difficulties in getting our desired estimate \eqref{Thm1r1}. This is not the case for lower-order derivatives; in fact, easier for $m=1$ and $2$ (see proof of \cite[Theorems 3.6, 4.6]{Sk15}). Perturbation operators appear at most twice in the expression of the lower-order derivatives ($m=1,2$), and due to the cyclicity property of the trace, for $m=1$, one may ignore the appearance of $V_j$. And for $m=2$, simple applications of Cauchy-Schwartz inequality provide the required estimate. In contrast, higher-order derivative expressions contain $m (\geq 3)$ number of perturbation operators. In that case, the cyclicity property of the trace and a direct application of  the Cauchy-Schwartz inequality will not help us to obtain the required estimate. As a result, we propose the following powerful tool to solve our problem,  which is outlined below.
\begin{lma}\normalfont{\cite[Theorem 4.1]{DySk09}\cite[Equation (17)]{BSIEOT96}}\label{Lem3}
Assume $N_j, i_j,K_j\in\N$, $m\geq 2$, and $V_{j}\in\mLL(\Mcal,\tr),$ for $1\le j\le m$. Let $(\delta_{i_j})_{1\le i_{j}\le K_j}$ be a partition of $\C^{N_j}$ and let $E_j(\cdot)$ be the spectral measure on $\C^{N_j}$ for $1\le j\le m$. Assume that either $m=2$ or $\tr$ is the standard trace. Then the following estimate holds
$$\sum_{i_1=1}^{K_1}\sum_{i_2=1}^{K_2}\cdots\sum_{i_m=1}^{K_m}\left|\tr\left(\prod_{j=1}^{m}\left(E_j(\delta_{i_j})V_j\right)\right)\right|\le\prod_{j=1}^{m}\left(\|V_{j}\|_{2}\right).$$
\end{lma}	

\begin{rmrk}
Note that, Lemma \ref{Lem3} holds only for $m\ge 2$. However, for $m=1,$ a special case has already been discussed in \cite[Lemma 3.5]{Sk15} by considering $V_{j}\in\mL(\Mcal,\tr),$ for $1\le j\le m.$ For the general trace $\tr,$ it is worth noting that Lemma \ref{Lem3} is no more valid for $m>2$ (See \cite[Remark 4.2]{DySk09}). 
\end{rmrk}

\begin{proof}[Proof of Theorem \ref{Thm1}]
The proof for $m=2$ is established in \cite[Theorem 4.6]{Sk15}. Assume $\tr$ is the standard trace and $m\geq 3$. We give the proof only for $\Omega\,=\,\cir$; the proof for $\Omega\,=\,[-1, 1]$ is analogous. Let $N\in\N$, and
$$f_{N}(z_1,\ldots,z_n)=\sum_{0\le k_1,\ldots,k_n\le N}^{}c_{k_1,\ldots,k_n}z_{1}^{k_{1}}\cdots z_{n}^{k_{n}}.$$
From \eqref{Lem2r10} we obtain
$$\left\|D_{f_{N}}^{j_{1}^{i_{1}},\ldots,j_{k}^{i_{k}}}(t)\right\|_{1}\le\left(\prod_{l=1}^{k}\|V_{j_{l}}\|_{2}^{i_{l}}\right)\sum_{0\le k_1,\ldots,k_n\le N}^{}\mathscr{K}_{k_{1},\ldots,k_{n}}\left|c_{k_1,\ldots,k_n}\right|,$$
and
$$\left\|D_{f}^{j_{1}^{i_{1}},\ldots,j_{k}^{i_{k}}}(t)\right\|_{1}\le\left(\prod_{l=1}^{k}\|V_{j_{l}}\|_{2}^{i_{l}}\right)\sum_{k_1,\ldots,k_n\ge 0}^{}\mathscr{K}_{k_{1},\ldots,k_{n}}\left|c_{k_1,\ldots,k_n}\right|,$$
where $\mathscr{K}_{k_{1},\ldots,k_{n}}$ is the co-efficient of \,$\frac{\partial^{m}}{\partial z_{j_{1}}^{i_{1}}\cdots \partial z_{j_{k}}^{i_{k}}}\left(z_{1}^{k_{1}}\ldots z_{n}^{k_{n}}\right)$.
As a result we have
$$\tr\left(D_{f}^{j_{1}^{i_{1}},\ldots,j_{k}^{i_{k}}}(t)\right)=\lim_{N\to\infty}\tr\left(D_{f_{N}}^{j_{1}^{i_{1}},\ldots,j_{k}^{i_{k}}}(t)\right).$$
Therefore it is sufficient to prove \eqref{Thm1r1} for the function $f_{N}$. Now \eqref{Lem1r2} and the pairwise commutativity of $X_{1}(t),\ldots,X_{n}(t)$ ensures that
\begin{align}
\nonumber&\tr\left(D_{f_{N}}^{j_{1}^{i_{1}},\ldots,j_{k}^{i_{k}}}(t)\right)\\
\nonumber&=\sum_{0\le k_1,\ldots,k_n\le N}^{}c_{k_1,\ldots,k_n}\,\Bigg[\Bigg(\prod_{l=1}^{k}i_{l}!\sum_{\substack {p_{l,0},\ldots,p_{l,i_{l}}\ge 0;\\p_{l,0}+\cdots+ p_{l,i_{l}}=k_{j_{l}}-i_{l}}}\Bigg)\tr\bigg\{T_{k_1,\ldots,k_{j_{1}-1}}(\bm{X}_{j_{1}-1}(t))\\ \nonumber&\hspace{1.4in}\times\prod_{q=1}^{k-1}\left(\prod_{l=0}^{i_{q}-1}\left((X_{j_{q}}(t))^{p_{q,l}}V_{j_{q}}\right)T_{p_{q,i_{q}},k_{j_{q}+1},\ldots,k_{j_{q+1}-1}}({}_ {j_{q}}\bm{X}_{j_{q+1}-1}(t))\right)\\
\nonumber&\hspace{1.4in}\times\left(\prod_{l=0}^{i_{k}-1}\left((X_{j_{k}}(t))^{p_{k,l}}V_{j_{k}}\right)\right)T_{p_{k,i_{k}},k_{j_{k}+1},\ldots,k_{n}}({}_ {j_{k}}\bm{X}_{n}(t))\bigg\}\Bigg]\\
\nonumber&=\sum_{0\le k_1,\ldots,k_n\le N}^{}c_{k_1,\ldots,k_n}\,\Bigg[\Bigg(\prod_{l=1}^{k}i_{l}!\sum_{\substack {p_{l,0},\ldots,p_{l,i_{l}}\ge 0;\\p_{l,0}+\cdots+ p_{l,i_{l}}=k_{j_{l}}-i_{l}}}\Bigg)\tr\bigg\{T_{k_1,\ldots,k_{j_{1}-1},p_{1,0}}(\bm{X}_{j_{1}}(t))\\
\nonumber&\hspace*{1.4in}\times T_{p_{k,i_{k}},k_{j_{k}+1},\ldots,k_{n}}({}_{j_{k}}\bm{X}_{n}(t))V_{j_{1}}\left(\prod_{l=1}^{i_{1}-1}\left((X_{j_{1}}(t))^{p_{1,l}}V_{j_{1}}\right)\right)\\
\label{Thm1r2}&\quad\times \prod_{q=1}^{k-1}\left(T_{p_{q,i_{q}},k_{j_{q}+1},\ldots,k_{j_{q+1}-1},p_{q+1,0}}({}_{j_{q}}\bm{X}_{j_{q+1}}(t))V_{j_{q+1}}\prod_{l=1}^{i_{q+1}-1}\left((X_{j_{q+1}}(t))^{p_{q+1,l}}V_{j_{q+1}}\right)\right)\bigg\}\Bigg],
\end{align}
where $2\le k\le m,~1\le j_{1}<\cdots<j_{k}\le n$ along with $1\le i_{1},\ldots,i_{k}\le m$ such that $\sum_{l=1}^{k}i_{l}=m$ and
\begin{align}
\nonumber\tr\left(D_{f_{N}}^{j^{m}}(t)\right)&=m!\sum_{0\le k_1,\ldots,k_n\le N}^{}c_{k_1,\ldots,k_n}\Bigg[\sum_{\substack {p_{0},\ldots,p_{m}\ge 0;\\ p_{0}+\cdots+p_{m}=k_{j}-m}}\\
\nonumber&\hspace*{0.5in}\tr\left\{T_{k_1,\ldots,k_{j-1},p_{0}}(\bm{X}_{j}(t))\,V_j\left(\prod_{l=1}^{m-1}\left((X_{j}(t))^{p_{l}}\,V_j\right)\right)T_{p_m,k_{j+1},\ldots,k_{n}}({}_{j}\bm{X}_{n}(t))\right\}\Bigg]\\
\nonumber&=m!\sum_{0\le k_1,\ldots,k_n\le N}^{}c_{k_1,\ldots,k_n}\Bigg[\sum_{\substack {p_{0},\ldots,p_{m}\ge 0;\\ p_{0}+\cdots+p_{m}=k_{j}-m}}\\
\label{Thm1r3}&\hspace*{0.5in}\tr\left\{T_{k_1,\ldots,k_{j-1}, p_{0}+p_{m}, k_{j+1},\ldots,k_{n}}(\bm{X}_{n}(t))\,V_{j}\left(\prod_{l=1}^{m-1}\left((X_{j}(t))^{p_{l}}\,V_j\right)\right)\right\}\Bigg].
\end{align}
\begin{case}\label{Thm1case1}
{\normalfont Assume that $\bm{X}_{n}(t)$ consists of commuting normal contractions.
\begin{subcase}
{\normalfont Consider the case when the differentiation occurs only at one variable, that is, $i_{l}=m$ and $i_{k}=0$ for all $k\neq l.$ Without loss of generality we assume $i_{1}=m$ and $j_{1}=j.$	
	
Let $E_t(\cdot)$ be the joint spectral measure of the tuple $(X_{1}(t),\ldots,X_{n}(t))$ and let $E_{t,j}(\cdot)$ be the spectral measure of $X_{j}(t)$. Then the spectral theorem ensures that
\begin{align}
\nonumber&X_{j}(t)^{p}=\int_{\C}z^{p}dE_{t,j}(z),
\end{align}
and
\begin{align}
\nonumber&T_{k_1,\ldots,k_{j-1}, p_{0}+p_{m}, k_{j+1},\ldots,k_{n}}(\bm{X}_{n}(t))=\int_{\C^{n}}z_{1}^{k_1}\cdots z_{j-1}^{k_{j-1}}z_{j}^{p_{0}+p_{m}}z_{j+1}^{k_{j+1}}\cdots z_{n}^{k_n}dE_{t}(z_1,\ldots,z_n).
\end{align}
Hence, for every $1\le l\le n$ and $1\le i\le m-1$, we have the existence of sequences of Borel partitions $(\delta_{s,l,\beta_{l}})_{1\le\beta_{l}\le s}$ and $({\widetilde{\delta}_{s,i,\alpha_{i}}})_{1\le\alpha_{i}\le s}$ of $\C$, and sequences of complex numbers $(z_{s,l,\beta_{l}})_{1\le l\le n,\,1\le\beta_{l}\le s}$ and  $({\widetilde{z}_{s,i,\alpha_{i}}})_{1\le\alpha_{i}\le s}$ respectively, such that
\begin{align}
\nonumber&\tr\left[T_{k_1,\ldots,k_{j-1}, p_{0}+p_{m}, k_{j+1},\ldots,k_{n}}(\bm{X}_{n}(t))\,V_{j}\left(\prod_{l=1}^{m-1}\left((X_{j}(t))^{p_{l}}\,V_j\right)\right)\right]\\
\nonumber&=\lim_{s\to\infty}\sum_{1\,\le\beta_{1},\ldots,\beta_{n},\alpha_{1},\ldots,\alpha_{m-1}\le s}\left(z_{s,1,\beta_{1}}^{k_{1}}\cdots z_{s,j-1,\beta_{j-1}}^{k_{j-1}}z_{s,j,\beta_{j}}^{p_{0}+p_{m}}\,z_{s,j+1,\beta_{j+1}}^{k_{j+1}}\cdots z_{s,n,\beta_{n}}^{k_{n}}\,{\widetilde{z}_{s,1,\alpha_{1}}}^{p_{1}}\cdots\,{\widetilde{z}_{s,m-1,\alpha_{m-1}}}^{p_{m-1}}\right)\\
\label{Thm1r4}&\hspace{1.7in}\times\tr\left[E_{t}(\delta_{s,1,\beta_{1}}\times\cdots\times\delta_{s,n,\beta_{n}})V_{j}\left(\prod_{l=1}^{m-1}\left(E_{t,j}({\widetilde{\delta}_{s,l,\alpha_{l}}})\,V_j\right)\right)\right].
\end{align}
Now using \eqref{Thm1r4} and \eqref{Lem6r2}, we derive
\begin{align}
\nonumber\tr\left(D_{f_{N}}^{j^{m}}(t)\right)&=m!\lim_{s\to\infty}\sum_{1\,\le\beta_{1},\ldots,\beta_{n},\alpha_{1},\ldots,\alpha_{m-1}\le s}\\
\nonumber&\quad f_{N}\bigg(z_{s,1,\beta_{1}},\ldots,z_{s,j-1,\beta_{j-1}},\left[{\widetilde{z}_{s,1,\alpha_{1}}},\ldots,{\widetilde{z}_{s,m-1,\alpha_{m-1}}},z_{s,j,\beta_{j}},z_{s,j,\beta_{j}}\right],z_{s,j+1,\beta_{j+1}},\ldots,z_{s,n,\beta_{n}}\bigg)\\
\label{Thm1r5}&\hspace{0.8in} \times\tr\left[E_{t}(\delta_{s,1,\beta_{1}}\times\cdots\times\delta_{s,n,\beta_{n}})V_{j}\left(\prod_{l=1}^{m-1}\left(E_{t,j}({\widetilde{\delta}_{s,l,\alpha_{l}}})\,V_j\right)\right)\right].
\end{align}
Therefore the application of Lemmas \ref{Lem5} and \ref{Lem3} to \eqref{Thm1r5} gives 
$$\left|\tr\left(D_{f}^{j^{m}}(t)\right)\right|\le(\|V_{j}\|_{2})^{m}\left\|\frac{\partial^{m}f}{\partial z_{j}^{m}}\right\|_{L^{\infty}(\cir^{n})}.$$}
\end{subcase}
\begin{subcase}\label{Thm1r6}
{\normalfont Consider the case when the differentiation occurs atleast at two variables.

Suppose that $j_{1}<j_{2}<\cdots<j_{k}$ and denote by $E_{t,j_{h},j_{h}+1,\ldots,j_{h+1}}(\cdot)$ the spectral measure defined as follows
$$E_{t,j_{h},j_{h}+1,\ldots,j_{h+1}}(S_{j_{h}},\ldots,S_{j_{h+1}})=E_{t}(\C\times\cdots\times\C\times S_{j_{h}}\times\cdots\times S_{j_{h+1}}\times\C\times\cdots\times\C),$$
where $S_{j_{h}}$ is a Borel subset of $\C$ for $1\le h\le k-1$, and $E_{t,j_{l}}(\cdot)$ is the spectral measure of $X_{j_{l}}(t)$ for $1\le l\le k$. Then for $1\le h\le k-1$, we have the following
\begin{align}
\nonumber&T_{p_{h,i_{h}},k_{j_{h}+1},\ldots,k_{j_{h+1}-1},p_{h+1,0}}({}_{j_{h}}\bm{X}_{j_{h+1}}(t))\\
\nonumber&=\int_{\C^{j_{h+1}+1-j_{h}}}^{}z_{j_{h}}^{p_{h,i_{h}}}z_{j_{h}+1}^{k_{j_{h}+1}}\cdots z_{j_{h+1}-1}^{k_{j_{h+1}-1}}z_{j_{h+1}}^{p_{h+1,0}}dE_{t,j_{h},j_{h}+1,\ldots,j_{h+1}}(z_{j_{h}},\ldots,z_{j_{h+1}}).
\end{align}
Consider $1\le l\le k$ and $1\leq h\leq k-1$. Then for every $1\le a\le i_{l}-1$ and $j_{h}\le\eta\le j_{h+1}$, we have the existence of  sequences of Borel partitions $({\widetilde{\delta}_{i_{l},s,a,\gamma_{l,a}}})_{1\le\gamma_{l,a}\le s}$  and $(\delta_{h,s,\eta,\alpha_{h,\eta}})_{1\le\alpha_{h,\eta}\le s}$ of $\C$, and  sequences of complex numbers $({\widetilde{z}_{i_{l},s,a,\gamma_{l,a}}})_{1\le\gamma_{l,a}\le s}$ and $(z_{h,s,\eta,\alpha_{h,\eta}})_{j_{h}\le\eta\le j_{h+1},\,1\le\alpha_{h,\eta}\le s}$ respectively, such that
\begin{align}
\nonumber&\tr\Bigg[T_{k_1,\ldots,k_{j_{1}-1},p_{1,0}}(\bm{X}_{j_{1}}(t))T_{p_{k,i_{k}},k_{j_{k}+1},\ldots,k_{n}}({}_{j_{k}}\bm{X}_{n}(t))V_{j_{1}}\left(\prod_{l=1}^{i_{1}-1}\left((X_{j_{1}}(t))^{p_{1,l}}V_{j_{1}}\right)\right)\\
\nonumber&\hspace{0.7in}\times\prod_{q=1}^{k-1}\left(T_{p_{q,i_{q}},k_{j_{q}+1},\ldots,k_{j_{q+1}-1},p_{q+1,0}}({}_{j_{q}}\bm{X}_{j_{q+1}}(t))V_{j_{q+1}}\prod_{l=1}^{i_{q+1}-1}\left((X_{j_{q+1}}(t))^{p_{q+1,l}}V_{j_{q+1}}\right)\right)\Bigg]\\
\nonumber&=\lim_{s\to\infty}\sum_{\substack {1\,\le\beta_{1},\ldots,\beta_{n}\le s;\\\alpha_{h,j_{h}},\alpha_{h,j_{h}+1},\ldots,\alpha_{h,j_{h+1}}\,\le\,s,~1\le h\le k-1;\\ 1\le\gamma_{l,1},\gamma_{l,2},\ldots,\gamma_{l,i_{l}-1}\le s,~1\le l\le k}}\\
\nonumber&\Bigg[\left(z_{s,1,\beta_{1}}^{k_{1}}\cdots z_{s,j_{1}-1,\beta_{j_{1}-1}}^{k_{j_{1}-1}}z_{s,j_{1},\beta_{j_{1}}}^{p_{1,0}}z_{s,j_{k},\beta_{j_{k}}}^{p_{k,i_{k}}}\cdots z_{s,n,\beta_{n}}^{k_{n}}\right)\left(\prod_{l=1}^{i_{1}-1}\left(\widetilde{z}_{i_{1},s,l,\gamma_{1,l}}\right)^{p_{1,l}}\right)\\
\nonumber&\hspace{0.7in}\times\prod_{q=1}^{k-1}\left\{z_{q,s,j_{q},\alpha_{q,j_{q}}}^{p_{q,i_{q}}}\hspace*{-0.1cm}\left(\prod_{r=1}^{j_{q+1}-j_{q}-1}z_{q,s,j_{q}+r,\alpha_{q,j_{q}+r}}^{k_{j_{q}+r}}\right)\hspace*{-0.1cm}z_{q,s,j_{q+1},\alpha_{q,j_{q+1}}}^{p_{q+1,0}}\hspace*{-0.2cm}\prod_{l=1}^{i_{q+1}-1}\left(\widetilde{z}_{i_{q+1},s,l,\gamma_{q+1,l}}\right)^{p_{q+1,l}}\right\}\Bigg]\\
\nonumber&\times\tr\Bigg[E_{t}\left(\delta_{s,1,\beta_{1}}\times\cdots\times\delta_{s,n,\beta_{n}}\right)\,V_{j_{1}}\left(\prod_{l=1}^{i_{1}-1}\left(E_{t,j_{1}}(\widetilde{\delta}_{i_{1},s,l,\gamma_{1,l}})V_{j_{1}}\right)\right)\\
\label{Thm1r7}&\times\prod_{q=1}^{k-1}\left\{E_{t,j_{q},j_{q}+1,\ldots,j_{q+1}}\left(\prod_{r=0}^{j_{q+1}-j_{q}}\delta_{q,s,j_{q}+r,\alpha_{q,j_{q}+r}}\right)V_{j_{q+1}}\prod_{l=1}^{i_{q+1}-1}\left(E_{t,j_{q+1}}(\widetilde{\delta}_{i_{q+1},s,l,\gamma_{q+1,l}})V_{j_{q+1}}\right)\right\}\Bigg].
\end{align}
Now applying \eqref{Lem6r4} to \eqref{Thm1r7}, we get
\begin{align}
\nonumber&\tr\left(D_{f_{N}}^{j_{1}^{i_{1}},\ldots,j_{k}^{i_{k}}}(t)\right)=\prod_{l=1}^{k}(i_{l}!)\lim_{s\to\infty}\sum_{\substack {1\,\le\beta_{1},\ldots,\beta_{n}\le s;\\\alpha_{h,j_{h}},\alpha_{h,j_{h}+1},\ldots,\alpha_{h,j_{h+1}}\,\le\,s,~1\le h\le k-1;\\ 1\le\gamma_{l,1},\gamma_{l,2},\ldots,\gamma_{l,i_{l}-1}\le s,~1\le l\le k}}\\
\nonumber&\Bigg\{f_{N}\Big(z_{s,1,\beta_{1}},\ldots,z_{s,j_{1}-1,\beta_{j_{1}-1}},\left[z_{s,j_{1},\beta_{j_{1}}},z_{1,s,j_{1},\alpha_{1,j_{1}}},\widetilde{z}_{i_{1},s,1,\gamma_{1,1}},\ldots,\widetilde{z}_{i_{1},s,i_{1}-1,\gamma_{1,i_{1}-1}}\right],z_{1,s,j_{1}+1,\alpha_{1,j_{1}+1}},\\
\nonumber&\ldots,z_{1,s,j_{2}-1,\alpha_{1,j_{2}-1}},\left[z_{1,s,j_{2},\alpha_{1,j_{2}}},z_{2,s,j_{2},\alpha_{2,j_{2}}},\widetilde{z}_{i_{2},s,1,\gamma_{2,1}},\ldots,\widetilde{z}_{i_{2},s,i_{2}-1,\gamma_{2,i_{2}-1}}\right],z_{2,s,j_{2}+1,\alpha_{2,j_{2}+1}},\ldots,\\
\nonumber&z_{k-1,s,j_{k}-1,\alpha_{k-1,j_{k}-1}},\hspace*{-0.1cm}\left[z_{k-1,s,j_{k},\alpha_{k-1,j_{k}}},z_{s,j_{k},\beta_{j_{k}}},\widetilde{z}_{i_{k},s,1,\gamma_{k,1}},\ldots,\widetilde{z}_{i_{k},s,i_{k}-1,\gamma_{k,i_{k}-1}}\right]\hspace*{-0.1cm},z_{s,j_{k}+1,\beta_{j_{k}+1}},\ldots,z_{s,n,\beta_{n}}\hspace*{-0.1cm}\Big)\hspace*{-0.15cm}\Bigg\}\\
\nonumber&\times\tr\Bigg[E_{t}\left(\delta_{s,1,\beta_{1}}\times\cdots\times\delta_{s,n,\beta_{n}}\right)\,V_{j_{1}}\left(\prod_{l=1}^{i_{1}-1}\left(E_{t,j_{1}}(\widetilde{\delta}_{i_{1},s,l,\gamma_{1,l}})V_{j_{1}}\right)\right)\\
\label{Thm1r8}&\times\prod_{q=1}^{k-1}\left\{E_{t,j_{q},j_{q}+1,\ldots,j_{q+1}}\left(\prod_{r=0}^{j_{q+1}-j_{q}}\delta_{q,s,j_{q}+r,\alpha_{q,j_{q}+r}}\right)V_{j_{q+1}}\prod_{l=1}^{i_{q+1}-1}\left(E_{t,j_{q+1}}(\widetilde{\delta}_{i_{q+1},s,l,\gamma_{q+1,l}})V_{j_{q+1}}\right)\right\}\Bigg].
\end{align}
Finally, applying Lemmas \ref{Lem5} and \ref{Lem3} to \eqref{Thm1r8} we obtain
\begin{flalign}
\nonumber&\left|\tr\left(D_{f}^{j_{1}^{i_{1}},\ldots,j_{k}^{i_{k}}}(t)\right)\right|\le\left(\prod_{l=1}^{k}\|V_{j_{l}}\|_{2}^{i_{l}}\right)\left\|\frac{\partial^{m}f}{\partial z_{j_{1}}^{i_{1}}\cdots\partial z_{j_{k}}^{i_{k}}}\right\|_{L^{\infty}(\cir^{n})}.
\end{flalign}
This completes the proof for the case $\bm{X}_{n}(t)$ is a tuple of commuting normal contractions.}
\end{subcase}}
\end{case}
\begin{case}
{\normalfont Assume that $\bm{X}_{n}(t)$ is only a tuple of commuting contractions having commuting normal dilation.
\vspace{0.1in}
	
\noindent Since there exists a tuple of commuting normal contractions $\bm{U}_{n}$ on a Hilbert space $\mathcal{K}\supset\hil$ such that it satisfies \eqref{Dilation}, then by using the cyclicity property of the trace we have
\begin{align}\label{Thm1new}
\nonumber&\tr\Bigg[T_{k_1,\ldots,k_{j_{1}-1},p_{1,0}}(\bm{X}_{j_{1}}(t))T_{p_{k,i_{k}},k_{j_{k}+1},\ldots,k_{n}}({}_{j_{k}}\bm{X}_{n}(t))V_{j_{1}}\left(\prod_{l=1}^{i_{1}-1}\left((X_{j_{1}}(t))^{p_{1,l}}V_{j_{1}}\right)\right)\\
\nonumber&\hspace{0.4in}\times \prod_{q=1}^{k-1}\left(T_{p_{q,i_{q}},k_{j_{q}+1},\ldots,k_{j_{q+1}-1},p_{q+1,0}}({}_{j_{q}}\bm{X}_{j_{q+1}}(t))V_{j_{q+1}}\prod_{l=1}^{i_{q+1}-1}\left((X_{j_{q+1}}(t))^{p_{q+1,l}}V_{j_{q+1}}\right)\right)\Bigg]\\
\nonumber&=\tr\Bigg[P_{\hil}T_{k_1,\ldots,k_{j_{1}-1},p_{1,0}}(\bm{U}_{j_{1}})T_{p_{k,i_{k}},k_{j_{k}+1},\ldots,k_{n}}({}_{j_{k}}\bm{U}_{n})V_{j_{1}}P_{\hil}\left(\prod_{l=1}^{i_{1}-1}\left((U_{j_{1}})^{p_{1,l}}V_{j_{1}}P_{\hil}\right)\right)\\
&\hspace{0.4in}\times \prod_{q=1}^{k-1}\left(T_{p_{q,i_{q}},k_{j_{q}+1},\ldots,k_{j_{q+1}-1},p_{q+1,0}}({}_{j_{q}}\bm{U}_{j_{q+1}})V_{j_{q+1}}P_{\hil}\prod_{l=1}^{i_{q+1}-1}\left((U_{j_{q+1}})^{p_{q+1,l}}V_{j_{q+1}}P_{\hil}\right)\right)\Bigg].
\end{align}
Therefore by Subcase \ref{Thm1r6}, from  \eqref{Thm1new} we conclude
\begin{align}
\nonumber\left|\tr\left(D_{f}^{j_{1}^{i_{1}},\ldots,j_{k}^{i_{k}}}(t)\right)\right|&\le\left(\prod_{l=1}^{k}\|V_{j_{l}}P_{\hil}\|_{2}^{i_{l}}\right)\left\|\frac{\partial^{m}f}{\partial z_{j_{1}}^{i_{1}}\cdots\partial z_{j_{k}}^{i_{k}}}\right\|_{L^{\infty}(\cir^{n})}\le\left(\prod_{l=1}^{k}\|V_{j_{l}}\|_{2}^{i_{l}}\right)\left\|\frac{\partial^{m}f}{\partial z_{j_{1}}^{i_{1}}\cdots\partial z_{j_{k}}^{i_{k}}}\right\|_{L^{\infty}(\cir^{n})}.
\end{align}
Similarly, we also derive the estimate for $\left|\tr\left(D_{f}^{j^{m}}(t)\right)\right|$. This completes the proof.}\vspace*{-0.7cm}
\end{case}
\end{proof}
In the next lemma, we extend the result of {\cite[Lemma 4.7]{Sk15}} by showing that the higher-order Taylor remainder can be represented as the integral representations in terms of the higher-order derivative.
\begin{lma}\label{Lem7}
Assume Notations \ref{Notations}. Let $V_{j}\in\mLL(\Mcal, \tr)$ for $1\le j\le n$. Then for every $f\,\in\mathcal{A}(\D),\,\epsilon>0$, and $m\ge 2$
\begin{equation}\label{Lem7r1}
\tr\left[f(\bm{B}_n)-\sum_{k=0}^{m-1}\frac{1}{k!}\,\frac{d^k}{ds^k}\bigg|_{s=0}f(\bm X_n(s))\right]=\int_{0}^{1}\frac{(1-t)^{m-1}}{(m-1)!}\tr\left(\dmds\bigg|_{s=t}^{}f(\bm{X}_n(s))\right)dt.
\end{equation}
\end{lma}	
\begin{proof}
The proof of \eqref{Lem7r1} goes by the induction on $m$. \cite[Lemma 4.7]{Sk15} provides the base of the induction for $m=2$. Suppose the formula \eqref{Lem7r1} holds for $m=r-1$, that is,
\begin{align}
\label{Lem7r2}&\tr\left[f(\bm{B}_n)-\sum_{k=0}^{r-2}\frac{1}{k!}\,\frac{d^k}{ds^k}\bigg|_{s=0}f(\bm X_n(s))\right]=\int_{0}^{1}\frac{(1-t)^{r-2}}{(r-2)!}\tr\left(\frac{d^{r-1}}{ds^{r-1}}\bigg|_{s=t}f(\bm{X}_n(s))\right)dt.	
\end{align}	
Since $\drrds\bigg|_{s=t}f(\bm{X}_n(s))\in\mL(\Mcal,\tr)$ for $r> 2,$ then from \eqref{Lem7r2} we derive
\begin{align}
\nonumber&\tr\left[f(\bm{B}_n)-\sum_{k=0}^{r-1}\frac{1}{k!}\,\frac{d^k}{ds^k}\bigg|_{s=0}f(\bm X_n(s))\right]\\
\label{Lem7r3}&=\int_{0}^{1}\frac{(1-t)^{r-2}}{(r-2)!}\tr\left(\drrds\bigg|_{s=t}^{}f(\bm{X}_n(s))\right)dt-\frac{1}{(r-1)!}\tr\left(\drrds\bigg|_{s=0}^{}f(\bm{X}_n(s))\right).
\end{align}	
Finally using Lemma \ref{Lem2}, we obtain
\begin{equation}\label{Lem7r4}
\ddt\tr\left(\drrds\bigg|_{s=t}^{}f(\bm{X}_n(s))\right)=\tr\left(\drds\bigg|_{s=t}^{}f(\bm{X}_n(s))\right).
\end{equation}
Therefore, by performing integration by parts in \eqref{Lem7r3}, we conclude \eqref{Lem7r1} for $m=r$. Therefore the result follows by the principle of mathematical induction. This completes the proof. 
\end{proof}
\begin{rmrk}
The method used in the proof of the above Lemma \ref{Lem7} is no longer valid for $m=2$. However it can be used for the case $m=2$ by considering $V_{j}\in\mL(\Mcal,\tr),~1\le j\le n$.
\end{rmrk}

The following is the main result in this article.
\begin{thm}\label{Thm2}
Assume Notations \ref{Notations}, $m\ge 2$ and $V_{j}\in\mLL(\Mcal,\tr)$ for $1\le j\le n$. Consider either $\tr$ is the standard trace or $m=2$. If $\{\bm{A}_{n}, \bm{B}_{n}\}\in\mathcal{N}_{n},$ then there exist finite measures $\mu_{j_{1}^{i_{1}}\cdots j_{k}^{i_{k}}}$ on $\Omega^{n}$ for $1\leq k\leq m$, $1\le j_1<\cdots< j_k \le n$, and $1\leq i_1,\ldots,i_k\leq m$ with $i_1+i_2+\cdots+i_k=m$  such that
\begin{align}
\label{Thm2r1}&\left\|\mu_{j_{1}^{i_{1}}\cdots j_{k}^{i_{k}}}\right\|\le\frac{1}{m!}\left(\prod_{l=1}^{k}\|V_{j_{l}}\|_{2}^{i_{l}}\right),
\end{align}
and
\begin{align}
\nonumber&\tr\left[f(\bm{B}_n)-\sum_{k=0}^{m-1}\frac{1}{k!}\,\frac{d^k}{ds^k}\bigg|_{s=0}f(\bm X_n(s))\right]=\\
\label{Thm2r3}&\sum_{k=1}^{m}\left[\sum_{1\le j_{1}<\cdots<j_{k}\le n}^{}\left(\sum_{\substack{i_{1},\ldots,i_{k}\ge 1;\\i_1+\cdots+i_{k}=m}}\frac{m!}{i_{1}!\cdots i_{k}!}\left(\int_{\Omega^{n}}\frac{\partial^{m}f}{\partial z_{j_{1}}^{i_{1}}\cdots\partial z_{j_{k}}^{i_{k}}}(z_1,\ldots,z_n)\,\,d\mu_{j_{1}^{i_{1}}\ldots j_{k}^{i_{k}}}(z_1,\ldots,z_n)\right)\right)\right],
\end{align}
for every $f\,\in\mathcal{A}(\D),\,\epsilon>0$, and $\Omega=\cir$.

In addition, if $\bm{A}_{n}$ and $\bm{B}_{n}$ are tuples of self-adjoint contractions, then there exist real-valued measures $\mu_{j_{1}^{i_{1}}\cdots j_{k}^{i_{k}}},$ such that \eqref{Thm2r1} and \eqref{Thm2r3} hold with $\Omega=[-1, 1]$ for $f\in\mathcal{A}(\Ra),\,\epsilon>0.$ 
\end{thm}	
\begin{proof}
From Lemmas \ref{Lem2} and \ref{Lem7} we obtain
\begin{align}
\nonumber&\tr\left[f(\bm{B}_n)-\sum_{k=0}^{m-1}\frac{1}{k!}\,\frac{d^k}{ds^k}\bigg|_{s=0}f(\bm X_n(s))\right]\\
\label{Thm2r4}&\quad=\sum_{k=1}^{m}\left[\sum_{1\le j_{1}<\cdots<j_{k}\le n}^{}\left(\sum_{\substack{i_{1},\ldots,i_{k}\ge 1;\\i_1+\cdots+i_{k}=m}}\frac{m!}{i_{1}!\cdots i_{k}!}\int_{0}^{1}\frac{(1-t)^{m-1}}{(m-1)!}\tr\left(D_{f}^{j_{1}^{i_{1}},\ldots,j_{k}^{i_{k}}}(t)\right)dt\right)\right].
\end{align}
Now consider the collection of linear functionals $$\phi_{j_{1}^{i_{1}}\ldots j_{k}^{i_{k}}}:\text{span}\left\{\frac{\partial^{m}f}{\partial z_{j_{1}}^{i_{1}}\cdots\partial z_{j_{k}}^{i_{k}}}\,:\,\,f\in\mathcal{A}(\D)\right\}\to\C$$ defined by 
\begin{align}\label{Thm2new}
\phi_{j_{1}^{i_{1}}\ldots j_{k}^{i_{k}}}\left(\frac{\partial^{m}f}{\partial z_{j_{1}}^{i_{1}}\cdots\partial z_{j_{k}}^{i_{k}}}\right)=\int_{0}^{1}\frac{(1-t)^{m-1}}{(m-1)!}\tr\left(D_{f}^{j_{1}^{i_{1}},\ldots,j_{k}^{i_{k}}}(t)\right)dt,
\end{align}
for $1\le k\le m,~1\le j_{1}<\cdots<j_{k}\le n$, and $1\le i_{1},\ldots,i_{k}\le m$ such that $\sum_{l=1}^{k}i_{l}=m.$ 
Hence by applying Theorem \ref{Thm1} to \eqref{Thm2new} we have
\begin{align}
\label{Thm2r5}&\left|\phi_{j_{1}^{i_{1}}\ldots j_{k}^{i_{k}}}\left(\frac{\partial^{m}f}{\partial z_{j_{1}}^{i_{1}}\cdots\partial z_{j_{k}}^{i_{k}}}\right)\right|\le\frac{1}{m!}\left(\prod_{l=1}^{k}\|V_{j_{l}}\|_{2}^{i_{l}}\right)\left\|\frac{\partial^{m}f}{\partial z_{j_{1}}^{i_{1}}\cdots\partial z_{j_{k}}^{i_{k}}}\right\|_{L^{\infty}(\cir^{n})}.
\end{align}	
Therefore, by using the Hahn-Banach theorem and the Riesz-Markov representation theorem for $C(\cir^{n}),$ we conclude the existence of measures $\mu_{j_{1}^{i_{1}}\cdots j_{k}^{i_{k}}}$ on $\cir^{n}$ corresponding to the functionals $\phi_{j_{1}^{i_{1}}\ldots j_{k}^{i_{k}}}$ satisfying \eqref{Thm2r1}, and
\begin{equation}\label{Thm2r6}
\phi_{j_{1}^{i_{1}}\ldots j_{k}^{i_{k}}}\left(\frac{\partial^{m}f}{\partial z_{j_{1}}^{i_{1}}\cdots\partial z_{j_{k}}^{i_{k}}}\right)=\int_{\cir^{n}}\frac{\partial^{m}f}{\partial z_{j_{1}}^{i_{1}}\cdots\partial z_{j_{k}}^{i_{k}}}(z_{1},\ldots,z_{n})\,d\mu_{j_{1}^{i_{1}}\cdots j_{k}^{i_{k}}}(z_{1},\ldots,z_{n}),
\end{equation}
for every $1\le k\le m,~1\le j_{1}<\cdots<j_{k}\le n$, and $1\le i_{1},\ldots,i_{k}\le m$ such that $\sum_{l=1}^{k}i_{l}=m.$
Finally, the application of \eqref{Thm2r6} to \eqref{Thm2r4} concludes \eqref{Thm2r3} for $\Omega=\cir$. The proof goes analogously for $\Omega=[-1, 1]$. This completes the proof. 
\end{proof}		
Below, using the spectral shift measures $\mu_{j^{m}}$ obtained in Theorem \ref{Thm2}, we relate them to the spectral shift function associated with the pair $(A_{j} ,B_{j})$ for $1\le j\le n.$
\begin{crl}\label{Crl1}
Assume Notations \ref{Notations}, $g\in\mathcal{A}(\mathbb{D}_{1+\epsilon}),~\epsilon>0$, and $V_{j}\in\mLL(\Mcal,\tr)$ for $1\le j\le n$. Consider $\tr$ is the standard trace. Then there exist finite measures $\mu_{j}^{m}$ on $\Omega^{n}$ for $1\le j\le n$, such that
\begin{align}
\label{Crl1r1}&\int_{\Omega^{n}}g(z_{j})d\mu_{j^{m}}(z_{1},\ldots,z_{n})=\int_{\Omega} g(z)\eta_{m,j,A_{j},B_{j}}(z)dz,
\end{align}
where $\eta_{m,j,A_{j},B_{j}}$ is the spectral shift function corresponding to the self-adjoint pair $(A_{j},B_{j})$ for $\Omega=[-1,1]$, and associated with the pair of contractions $(A_{j},B_{j})$ for $\Omega=\cir$ respectively.
\end{crl}
\begin{proof}
Let $\Omega=\cir$. Consider $f(z_{1},\ldots,z_{n})=f(z_{j})$ satisfying $\frac{\partial^{m}f}{\partial z_{j}^{m}}=g.$ Then \eqref{Thm2r3} reduces to 
\begin{align}
\label{Crl1r2}&\tr\left[f(B_{j})-\sum_{k=0}^{m-1}\frac{1}{k!}\,\frac{d^k}{ds^k}\bigg|_{s=0}f(X_j(s))\right]=\int_{\cir^{n}}g(z_{j})\,d\mu_{j}^{m}(z_1,\ldots,z_n).
\end{align}
By \cite[Theorem 1.3]{PoSkSu14}, there exists a function $\eta_{m,j,A_{j},B_{j}}\in L^{1}(\cir)$ corresponding to the pair of contractions $(A_{j},B_{j})$ such that 
\begin{align}
\nonumber&\int_{\cir^{n}}g(z_{j})d\mu_{j}^{m}(z_{1},\ldots,z_{n})=\int_{\cir} g(z)\eta_{m,j,A_{j},B_{j}}(z)dz.
\end{align}
Similar argument also holds for the case $\Omega=[-1,1]$, where the spectral shift function $\eta_{m,j,A_{j},B_{j}}$ is provided by \cite[Theorem 1.1]{PoSkSu13}.
\end{proof}

\begin{rmrk}
Although our main results Theorem \ref{Thm1} and Theorem \ref{Thm2} are obtained for commuting $n$-tuple of contractions but the results can be improved for any commuting $n$-tuple of bounded operators associated with the function class $\mathcal{A}(\mathbb{D}^{n}_{r+\epsilon}),~\epsilon>0$, where $\mathbb{D}_{r+\epsilon}^{n}$ is a polydisc with radius $(r+\epsilon)$, where $r=\max_{1\le j\le n,\,t\in[0,1]}\|X_{j}(t)\|.$ The restrictions on $m$ for the general trace $\tr$ is purely technical. The main problem we face here is that, when $\tr$ is the general trace and $m>2$, then the Lemma \ref{Lem3} is no longer valid. As a consequence, one needs some new ideas to extend the results of  Theorems \ref{Thm1} and \ref{Thm2} for general trace $\tr$ with perturbation in the Schatten ideal $\mathcal{L}_{p}(\Mcal,\tr)$ for $p\in[1,\infty)$. We would also like to mention that, our techniques are unable to conclude the absolute continuity of our obtained measures in Theorem~\ref{Thm2} with respect to the Lebesgue measure, and we leave this as a subject of future investigation. 
\end{rmrk}

\section*{Acknowledgements}
\textit{The first and third authors gratefully acknowledge the support provided by IIT Guwahati, Government of India. The second author delightedly acknowledges the support provided by the Prime Minister's Research Fellowship (PMRF),  Government of India.}

\end{document}